\documentclass[11pt]{amsart}

\usepackage{xy}
\usepackage{array}
\usepackage{amssymb}
\usepackage{color}
\usepackage{psfrag}
\usepackage{graphicx}

\newcounter{mbcommc}
\newcounter{rmcommc}

\xyoption{all}

\newtheorem{lemma}{Lemma}[section]
\newtheorem{theorem}[lemma]{Theorem}
\newtheorem{corollary}[lemma]{Corollary}

\newtheorem{prop}[lemma]{Proposition}
\newtheorem{example}[lemma]{Example}
\newtheorem{remark}[lemma]{Remark}
\newtheorem{definition}[lemma]{Definition}

\newcommand{\Cox}{\operatorname{Cox}}

\newcommand{\op}{{\textit{op}}}

\newcommand{\id}{\operatorname{id}}
\newcommand{\B}{\mathcal{B}}
\newcommand{\G}{\mathcal{G}}

\newcommand{\HVCenter}[1]{\setbox0=\hbox{#1}%
  \dimen0=\wd0%
  \dimen1=\ht0%
  \divide\dimen0 by 2%
  \divide\dimen1 by 2%
  \hskip -\dimen0%
  \lower \dimen1%
  \box0%
  \hskip -\dimen0}
\newcommand{\HBCenter}[1]{\setbox0=\hbox{#1}%
  \dimen0=\wd0%
  \dimen1=\ht0%
  \divide\dimen0 by 2%
  \divide\dimen1 by 2%
  \hskip -\dimen0%
  \box0%
  \hskip -\dimen0}
\newcommand{\HTCenter}[1]{\setbox0=\hbox{#1}%
  \dimen0=\wd0%
  \dimen1=\ht0%
  \divide\dimen0 by 2%
  \hskip -\dimen0%
  \lower \dimen1%
  \box0%
  \hskip -\dimen0}
\newcommand{\RBCenter}[1]{\setbox0=\hbox{#1}%
  \dimen0=\wd0%
  \dimen1=\ht0%
  \hskip -\dimen0%
  \box0%
  \hskip -\dimen0}
\newcommand{\LTCenter}[1]{\setbox0=\hbox{#1}%
  \dimen0=\wd0%
  \dimen1=\ht0%
  \lower \dimen1%
  \box0%
  \hskip -\dimen0}
\newcommand{\RTCenter}[1]{\setbox0=\hbox{#1}%
  \dimen0=\wd0%
  \dimen1=\ht0%
  \hskip -\dimen0%
  \lower \dimen1%
  \box0%
  \hskip -\dimen0}

\begin{document}

\title[Reflection group presentations arising from cluster algebras]{
Reflection group presentations arising from cluster algebras}

\author[Barot]{Michael Barot}
\address{
Instituto de Matem\'{a}ticas
Universidad Nacional Aut\'{o}noma de M\'{e}xico
Ciudad Universitaria, M\'{e}xico, Distrito Federal, C.P. 04510
M\'{e}xico
}
\email{barot@matem.unam.mx}

\author[Marsh]{Robert J. Marsh}
\address{School of Mathematics \\
University of Leeds \\
Leeds LS2 9JT \\
England
}
\email{marsh@maths.leeds.ac.uk}



\begin{abstract}
We give a presentation of a finite crystallographic reflection group in
terms of an arbitrary seed in the corresponding cluster algebra
of finite type and interpret the presentation in terms of companion bases
in the associated root system.
\end{abstract}

\date{1 March 2013}

\thanks{
This work was supported by
DGAPA, Universidad Nacional Aut\'onoma de M\'exico,
the Engineering and Physical Sciences Research Council
[grant number EP/G007497/1] and the Institute for Mathematical
Research (FIM, Forschungsinstitut f\"{u}r Mathematik) at the
ETH, Z\"{u}rich.
}

\maketitle

\section*{Introduction}
The article~\cite{fominzelevinsky2} proved that cluster algebras of finite type
can be classified by the Dynkin diagrams. A matrix (the \emph{Cartan counterpart})
is associated to each seed in such a cluster algebra, and the cluster algebra
associated to a Dynkin diagram $\Delta$ is characterised by the fact that it contains
seeds whose Cartan counterparts coincide with the Cartan matrix of $\Delta$
(up to simultaneous permutation of the rows and columns).

It is well known that the Dynkin diagrams also classify the finite crystallographic
reflection groups (see e.g.~\cite{humphreys}). Our main result is a presentation of
the reflection group associated to a Dynkin diagram in terms of an arbitrary
seed in the corresponding cluster algebra.
The presentation is given in terms of an edge-weighted oriented graph, known as
a \emph{diagram}, which is associated by Fomin and Zelevinsky to the seed.

If the Cartan counterpart of the seed is the Cartan matrix, then our presentation
reduces to the usual Coxeter presentation of the reflection group. However, in general,
the diagram of a seed contains cycles, and the presentation includes extra relations
arising from the chordless cycles in the diagram (i.e.\ induced subgraphs which are
cycles).

More precisely, if $\Gamma$ is the diagram associated to a seed in a cluster algebra
of finite type, we define, for vertices $i,j$ of $\Gamma$,
$$m_{ij}=
\begin{cases}
2 & \text{if $i$ and $j$ are not connected;} \\
3 & \text{if $i$ and $j$ are connected by an edge of weight $1$;} \\
4 & \text{if $i$ and $j$ are connected by an edge of weight $2$;} \\
6 & \text{if $i$ and $j$ are connected by an edge of weight $3$.}
\end{cases}
$$
Let $W(\Gamma)$ to be the group with generators $s_i$, $i$ a vertex of $\Gamma$,
subject to the following relations (where $e$ denotes the identity element):
\begin{enumerate}
\item[(1)] $s_i^2=e$ for all $i$;
\item[(2)] $(s_is_j)^{m_{ij}}=e$ for all $i\not=j$;
\item[(3)] For any chordless cycle $C$ in $\Gamma$:
$$i_0\xrightarrow{w_1} i_1\xrightarrow{w_2} \cdots \rightarrow i_{d-1}\xrightarrow{w_0} i_0,$$
where either all of the weights are $1$, or $w_0=2$, we have:
$$(s_{i_0}s_{i_{1}}\cdots s_{i_{d-2}}s_{i_{d-1}}s_{i_{d-2}}\cdots s_{i_{1}})^2=e.$$
\end{enumerate}

\noindent \textbf{Theorem A} \,\,
\emph{Let $\Gamma$ be the diagram associated to a seed in a cluster algebra of finite type.
Then $W(\Gamma)$ is isomorphic to the corresponding reflection group.}
\vskip 0.3cm
This work was mainly motivated by the notions of quasi-Cartan
companion~\cite{bgz} and companion basis~\cite[Defn.\ 4.1]{parsons}
(see also~\cite{parsonsthesis}). A companion basis for a skew-symmetric
matrix $B$ arising in a seed of the cluster algebra is a set $S$ of roots in the
corresponding root system which form a basis for the root lattice and for which the
matrix $A$ whose entries are $(\alpha,\beta^{\vee})$ for $\alpha,\beta\in S$
is a positive quasi-Cartan companion matrix for $B$ in the sense of~\cite{bgz}.
This means that the diagonal entries are equal to $2$, that the values of the off-diagonal
entries in $A$ coincide with the corresponding entries of $B$ up to sign,
and that the symmetrisation of the matrix is positive definite. Quasi-Cartan
companions are used in~\cite{bgz} to give criteria for whether a cluster
algebra is of finite type and companion bases were used
in~\cite{parsonsthesis,parsons} to describe the dimension vectors of indecomposable modules
over cluster-tilted algebras in type $A$ (dimension vectors of indecomposable
modules over cluster-tilted algebras in the general concealed case were also independently
calculated later in~\cite{ringel}).

The reflections corresponding to the elements of a companion basis associated to
a seed in a finite type cluster algebra generate the corresponding
reflection group. We show that the generators in a presentation of the
reflection group as considered here can be regarded as this set of reflections.
This perspective also allows us to show that, in the simply-laced
case, the presentations obtained here coincide with root basis
presentations of crystallographic reflection groups (amongst others) which
were obtained in~\cite{cst} using signed graphs.

We were also motivated by the article~\cite{bm}, which gives a presentation of a complex
semisimple Lie algebra of simply-laced Dynkin type $\Delta$ in terms of any positive definite unit form
of type $\Delta$ (i.e.\ equivalent to the unit form associated to $\Delta$).
This presentation also contains relations arising from chordless cycles, in this case
arising in the bigraph describing the unit form.

The article is organised as follows. In Section 1, we recall some of the theory of
cluster algebras of finite type from~\cite{fominzelevinsky2}. In Section 2, we collect
together some of the properties of the seeds of such cluster algebras. In particular,
we consider the diagrams associated to such seeds in~\cite{fominzelevinsky2} and the
way in which they mutate, as well as the oriented cycles appearing in them.
In Section 3 we associate a group to an arbitrary seed in the cluster algebra, using
generators and relations.
In Section 4 we find an efficient subset of the relations sufficient
to define the group (this is the set of relations given above).
In Section 5 we show that the group defined is invariant (up
to isomorphism) under mutation of the seed.
Since the group associated to a seed whose Cartan counterpart
is a Cartan matrix is the reflection group corresponding to the cluster algebra,
it follows that the group associated to any seed is isomorphic to this reflection
group. In Section 6 we explain the connections with quasi-Cartan companion
matrices and companion bases and with the extended Coxeter groups
associated to signed graphs in~\cite{cst}. This also gives an alternative proof
of the main result in the simply-laced case.


\section{Cluster algebras of finite type}
We first recall the definition of a (skew-symmetrisable)
\emph{cluster algebra} without coefficients, as defined
by Fomin-Zelevinsky~\cite{fominzelevinsky1}.

Let $\mathbb{F}=\mathbb{Q}(u_1,u_2,\ldots ,u_n)$ be the field of rational
functions in $n$ indeterminates over $\mathbb{Q}$.
Recall that a square matrix $B$ is said to be \emph{skew-symmetrisable}
if there is a diagonal matrix $D$ of the same size with positive diagonal
entries such that $DB$ is skew-symmetric.
We consider pairs $(\mathbf{x},B)$, where
$\mathbf{x}=\{x_1,x_2,\ldots ,x_n\}$ is a free generating set of $\mathbb{F}$
and $B$ is an $n\times n$ skew-symmetrisable matrix.

Given $k\in [1,n]=\{1,2,\ldots ,n\}$, a new such pair
$\mu_k(\mathbf{x},B)=(\mathbf{x}',B')$, the \emph{mutation} of
$(\mathbf{x},B)$ at $k$, is defined as follows.
We have:
$$B'_{ij}=
\begin{cases}
-B_{ij} & \text{if $i=k$ or $j=k$;} \\
B_{ij}+\frac{|B_{ik}|B_{kj}+B_{ik}|B_{kj}|}{2} & \text{otherwise.}
\end{cases}
$$
Define an element $x'_k\in \mathbb{F}$ by the \emph{exchange relation}
$$x'_kx_k=\prod_{B_{ik}>0}x_i^{B_{ik}}+\prod_{B_{ik}<0}x_i^{-B_{ik}},$$
and set $\mathbf{x}'=\{x_1,x_2,\ldots ,x_{k-1},x'_k,x_{k+1},\ldots ,x_n\}$.
We write $\mu_k(B)=B'$.

Now we fix a pair $(\mathbf{x},B)$, the \emph{initial seed}.
We shall call any pair that can be obtained from
$(\mathbf{x},B)$ by a sequence of mutations a \emph{seed}.
Two seeds are regarded as the same
if there is a permutation of $[1,n]$ which takes
the matrix in the first seed to the second (when regarded as a simultaneous
permutation of the rows and columns) and the free generating set in the first
to the free generating set in the second.
Note that mutating twice at $k$ sends a seed to itself.

Each tuple $\mathbf{x}=\{x_1,\ldots,x_n\}$ such that there
exists a seed $(\mathbf{x},B)$ is called a \emph{cluster} and
its entries $x_i$ are called \emph{cluster variables}.
The cluster algebra $\mathcal{A}(\mathbf{x},B)$ is the
$\mathbb{Q}$-subalgebra of $\mathbb{F}$
generated by the set of all cluster variables.
Up to strong isomorphism (in the sense of~\cite[1.2]{fominzelevinsky2}),
$\mathcal{A}(\mathbf{x},B)$ does not depend on the choice of free
generating set $\mathbf{x}$ so we shall write $\mathcal{A}(B)$ for
$\mathcal{A}(\mathbf{x},B)$.

A cluster algebra is said to be of \emph{finite type} if it has finitely
many seeds. Given a skew-symmetrisable matrix $B$, its \emph{Cartan counterpart}
$A(B)$ is a matrix of the same size defined by

$$A_{ij}=\begin{cases}
2 & \text{if $i=j$;} \\
-|B_{ij}| & \text{otherwise.}
\end{cases}$$

We recall a special case of a result of Fomin-Zelevinsky. Note that the
original result deals with more general cluster algebras (with coefficients)
and allows more general matrices (sign-skew-symmetric matrices) in the
seeds.

\begin{theorem} \cite[1.4]{fominzelevinsky2}
There is a canonical bijection between the strong isomorphism classes of
skew-symmetrisable cluster algebras of finite type and the Cartan matrices
of finite type. A cluster algebra corresponds to a given Cartan matrix
$A$ provided it contains a seed for which the Cartan counterpart of the
matrix is $A$.
\end{theorem}

A skew-symmetrisable matrix $B$ is said to be \emph{$2$-finite} provided
$|B_{ij}B_{ji}|\leq 3$ for all $i,j\in [1,n]$. Note that this means the
opposite pair of entries $(B_{ij},B_{ji})$ in $B$ must be one of
$(1,1)$, $(1,2)$, $(2,1)$, $(1,3)$ or $(3,1)$. Then Fomin-Zelevinsky
also show that:

\begin{theorem} \cite[1.8]{fominzelevinsky2}
A cluster algebra $\mathcal{A}$ is of finite type if and only if
all the matrices appearing in the seeds of $\mathcal{A}$ are $2$-finite.
\end{theorem}

As observed by Fomin-Zelevinsky~\cite[\S9]{fominzelevinsky2}.

\begin{corollary} \label{c:submatrix}
Let $B$ be a $2$-finite matrix. Then any square submatrix of $B$, obtained
by taking a subset of the rows and the corresponding subset of the
columns, is again $2$-finite. Thus, passing to such a submatrix preserves
the fact that the corresponding cluster algebra is of finite type.
\end{corollary}

Fomin-Zelevinsky~\cite[7.3]{fominzelevinsky2}
define the \emph{diagram} $\Gamma(B)$ of a skew-symmetrisable
matrix $B$ to be the weighted oriented graph with vertices $[1,n]$.
There is an arrow from $i$ to $j$ if and only if $B_{ij}>0$, weighted
with the positive integer $|B_{ij}B_{ji}|$.
If $B$ is $2$-finite, then by~\cite[7.5]{fominzelevinsky2},
all $3$-cycles in the underlying unoriented graph of $\Gamma(B)$ are oriented
cyclically. In this case, the mutation rule for diagrams is as follows:

\begin{prop} \cite[Lemma 8.5]{fominzelevinsky2} \label{p:diagrammutation}
Let $B$ be a $2$-finite
skew-symmetrisable matrix. Then $\Gamma(\mu_k(B))$ is uniquely
determined by $\Gamma(B)$ as follows:
\begin{enumerate}
\item[(a)]
Reverse the orientations of all edges in $\Gamma(B)$ incident with
$k$ (leaving the weights unchanged).
\item[(b)]
For any path in $\Gamma(B)$ of form $i\overset{a}{\to} k\overset{b}{\to} j$
(i.e.\ with $a,b$ positive), let $c$ be the
weight on the edge $j\to i$, taken to be
zero if there is no such arrow. Let $c'$ be determined
by $c'\geq 0$ and
$$c+c'=max(a,b).$$
Then $\Gamma(B)$ changes as in Figure~\ref{f:diagrammutation}, again
taking the case $c'=0$ to mean no arrow between $i$ and $j$.
\end{enumerate}
\end{prop}

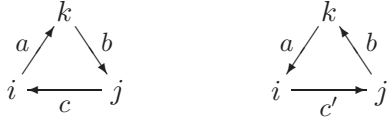
\begin{figure}[!ht]
\begin{picture}(140,40)
  \put(0,5){
    \multiput(0,0)(100,0){2}{
      \put(0,0){\HVCenter{$i$}}
      \put(40,0){\HVCenter{$j$}}
      \put(20,30){\HVCenter{$k$}}
      \put(4,15){\HBCenter{\small $a$}}
      \put(36,15){\HBCenter{\small $b$}}
    }
    \put(0,0){\
      \put(34,0){\vector(-1,0){28}}
      \put(4,6){\vector(2,3){12}}
      \put(24,24){\vector(2,-3){12}}
      \put(20,-6){\HVCenter{\small $c$}}
    }
    \put(100,0){\
      \put(6,0){\vector(1,0){28}}
      \put(16,24){\vector(-2,-3){12}}
      \put(36,6){\vector(-2,3){12}}
      \put(20,-6){\HVCenter{\small $c'$}}
    }
  }
\end{picture}
\caption{Mutation of a diagram.}
\label{f:diagrammutation}
\end{figure}

Note that the formula
$$\sqrt{c}+\sqrt{c'}=\sqrt{ab}$$
appears in~\cite[Lemma 8.5]{fominzelevinsky2} but it is easy to check
using~\cite[Prop.\ 9.7]{fominzelevinsky2} (see Proposition~\ref{p:cycles} below)
that the above formula gives the same answer. We thank the referee for pointing
this out to us.

We call $\Gamma(\mu_k(B))$ the \emph{mutation} of $\Gamma$ at $B$
and use the notation $\mu_k(\Gamma)$.

We say that a skew-symmetrisable matrix $B$ (and the corresponding diagram) is
\emph{of finite type} if the corresponding cluster algebra is of finite type.
Hence, by Corollary~\ref{c:submatrix}, we have:

\begin{corollary} \cite[\S9]{fominzelevinsky2}
Let $\Gamma$ be a diagram of finite type. Then any induced subdiagram
of $\Gamma$ is again of finite type.
\end{corollary}

\section{Diagrams of finite type}
In this section, we collect together some properties of diagrams of
finite type. We call an induced subgraph of an unoriented graph which is a
cycle a \emph{chordless cycle}. We first recall:

\begin{prop} \label{p:cycles} \cite[Prop.\ 9.7]{fominzelevinsky2}
Let $\Gamma$ be a diagram of finite type. Then a chordless cycle in the
underlying unoriented graph of $\Gamma$ is always cyclically oriented
in $\Gamma$. In addition, the unoriented graph underlying the cycle must be
one of those in Figure~\ref{f:cycles}.
\end{prop}

We will refer to such cycles as chordless cycles in $\Gamma$.

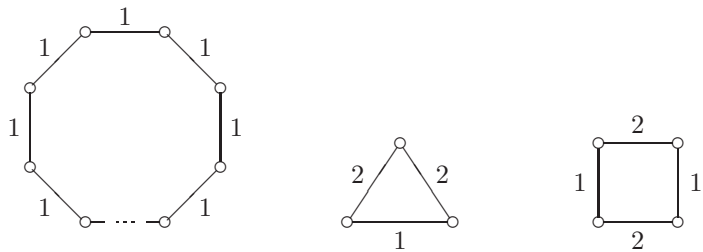
\begin{figure}[!ht]
\begin{picture}(270,88)
\put(0,8){
  \put(36,36){
    \multiput(-36,-15)(72,0){2}{\circle{4}}
    \multiput(-36,15)(72,0){2}{\circle{4}}
    \multiput(-15,-36)(0,72){2}{\circle{4}}
    \multiput(15,-36)(0,72){2}{\circle{4}}
    \multiput(-36,-13)(72,0){2}{\line(0,1){26}}
    \put(-13,36){\line(1,0){26}}
    \multiput(-15,36)(51,-51){2}{\put(-1.42,-1.42){\line(-1,-1){18.6}}}
    \multiput(-15,-36)(51,51){2}{\put(-1.42,1.42){\line(-1,1){18.6}}}
    \put(-13,-36){\line(1,0){5}}
    \put(13,-36){\line(-1,0){5}}
    \multiput(-3.5,-36)(3,0){3}{\line(1,0){1}}
    \put(-42,0){\HVCenter{\small $1$}}
    \put(42,0){\HVCenter{\small $1$}}
    \put(0,42){\HVCenter{\small $1$}}
    \put(30.4,30.4){\HVCenter{\small $1$}}
    \put(30.4,-30.4){\HVCenter{\small $1$}}
    \put(-30.4,30.4){\HVCenter{\small $1$}}
    \put(-30.4,-30.4){\HVCenter{\small $1$}}
  }
  \put(140,0){
    \multiput(-20,0)(40,0){2}{\circle{4}}
    \put(0,30){\circle{4}}
    \put(-18,0){\line(1,0){36}}
    \put(-18.8,1.8){\line(2,3){17.6}}
    \put(18.8,1.8){\line(-2,3){17.6}}
    \put(0,-7){\HVCenter{\small $1$}}
    \put(-16,15){\HBCenter{\small $2$}}
    \put(16,15){\HBCenter{\small $2$}}
  }
  \put(230,15){
    \multiput(-15,-15)(30,0){2}{\circle{4}}
    \multiput(-15,15)(30,0){2}{\circle{4}}
    \multiput(-15,-13)(30,0){2}{\line(0,1){26}}
    \multiput(-13,-15)(0,30){2}{\line(1,0){26}}
    \put(0,-22){\HVCenter{\small $2$}}
    \put(0,22){\HVCenter{\small $2$}}
    \put(-22,0){\HVCenter{\small $1$}}
    \put(22,0){\HVCenter{\small $1$}}
  }
}
\end{picture}
\caption{The chordless cycles in a diagram of finite type.}
\label{f:cycles}
\end{figure}

\begin{figure}[!ht]
\begin{picture}(310,43)
\put(0,5){
  \multiput(20,0)(90,0){4}{
    \multiput(-20,0)(40,0){2}{\circle{4}}
    \put(0,30){\circle{4}}
    \put(-18.8,1.8){\line(2,3){17.6}}
    \put(18.8,1.8){\line(-2,3){17.6}}
  }
  \put(20,0){
    \put(-16,15){\HBCenter{\small $1$}}
    \put(16,15){\HBCenter{\small $1$}}
  }
  \put(110,0){
    \put(0,-7){\HVCenter{\small $1$}}
    \put(-16,15){\HBCenter{\small $1$}}
    \put(16,15){\HBCenter{\small $1$}}
    \put(-18,0){\line(1,0){36}}
  }
  \put(200,0){
    \put(-16,15){\HBCenter{\small $1$}}
    \put(16,15){\HBCenter{\small $2$}}
  }
  \put(290,0){
    \put(0,-7){\HVCenter{\small $1$}}
    \put(-16,15){\HBCenter{\small $2$}}
    \put(16,15){\HBCenter{\small $2$}}
    \put(-18,0){\line(1,0){36}}
  }
}
\end{picture}
\caption{$3$-vertex connected subdiagrams of diagrams of finite type.}
\label{f:3vertex}
\end{figure}
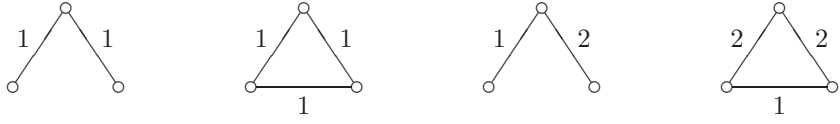

\begin{lemma} \label{l:3vertex}
Let $\Gamma$ be a diagram of finite type and let $\Xi$ be an
induced subdiagram with three vertices which is connected.
Then the unoriented graph underlying $\Xi$
must be one of those in Figure~\ref{f:3vertex}.
\end{lemma}

\begin{proof}
This follows from Proposition~\ref{p:cycles}, together with the fact
that a diagram of form $1\to 2\to 3$ (with positive weights)
cannot have a weight greater than $2$ or two weights equal to $2$,
by~\cite[Prop.\ 9.3]{fominzelevinsky2} and its proof.
\end{proof}

\begin{corollary} \label{c:localmutation}
Let $\Gamma$ be a graph of finite type and suppose that
are $3$ distinct vertices $i,j,k$ of $\Gamma$ with both $i$ and $j$
connected to $k$.
Then the effect of the mutation of $\Gamma$ at $k$ on the induced
subdiagram must be as in Figure~\ref{f:localmutation}
(starting on one side or the other), up to switching $i$ and $j$.
\end{corollary}

\begin{proof}
This follows from Lemma~\ref{l:3vertex} and Proposition~\ref{p:diagrammutation}.
\end{proof}

\begin{figure}[!ht]
\begin{picture}(350,170)
\put(10,5){
  \put(20,0){
    \put(-42,30){\HVCenter{(e)}}
    \put(0,0){
      \put(-20,0){\circle{4}}
      \put(20,0){\circle{4}}
      \put(0,30){\circle*{4}}
      \put(-28,0){\HVCenter{$i$}}
      \put(28,0){\HVCenter{$j$}}
      \put(0,34){\HBCenter{$k$}}
      \put(-18.2,2.7){\vector(2,3){16.4}}
      \put(1.8,27.3){\vector(2,-3){16.4}}
      \put(-16,15){\HBCenter{\scriptsize $1$}}
      \put(16,15){\HBCenter{\scriptsize $2$}}
    }
    \put(45,16){
      \put(0,0){\HVCenter{$\longleftrightarrow$}}
      \put(0,6){\HBCenter{$\mu_k$}}
    }
    \put(90,0){
      \put(-20,0){\circle{4}}
      \put(20,0){\circle{4}}
      \put(0,30){\circle*{4}}
      \put(-28,0){\HVCenter{$i$}}
      \put(28,0){\HVCenter{$j$}}
      \put(0,34){\HBCenter{$k$}}
      \put(-1.8,27.3){\vector(-2,-3){16.4}}
      \put(18.2,2.7){\vector(-2,3){16.4}}
      \put(-17,0){\vector(1,0){34}}
      \put(-16,15){\HBCenter{\scriptsize $1$}}
      \put(16,15){\HBCenter{\scriptsize $2$}}
      \put(0,-7){\HVCenter{\scriptsize $2$}}
    }
  }
  \put(240,0){
    \put(-42,30){\HVCenter{(f)}}
    \put(0,0){
      \put(-20,0){\circle{4}}
      \put(20,0){\circle{4}}
      \put(0,30){\circle*{4}}
      \put(-28,0){\HVCenter{$i$}}
      \put(28,0){\HVCenter{$j$}}
      \put(0,34){\HBCenter{$k$}}
      \put(-18.2,2.4){\vector(2,3){16.4}}
      \put(1.8,27.2){\vector(2,-3){16.4}}
      \put(17,0){\vector(-1,0){34}}
      \put(-16,15){\HBCenter{\scriptsize $2$}}
      \put(16,15){\HBCenter{\scriptsize $2$}}
      \put(0,-7){\HVCenter{\scriptsize $1$}}
    }
    \put(45,16){
      \put(0,0){\HVCenter{$\longleftrightarrow$}}
      \put(0,6){\HBCenter{$\mu_k$}}
    }
    \put(90,0){
      \put(-20,0){\circle{4}}
      \put(20,0){\circle{4}}
      \put(0,30){\circle*{4}}
      \put(-28,0){\HVCenter{$i$}}
      \put(28,0){\HVCenter{$j$}}
      \put(0,34){\HBCenter{$k$}}
      \put(-1.8,27.2){\vector(-2,-3){16.4}}
      \put(18.2,2.7){\vector(-2,3){16.4}}
      \put(-17,0){\vector(1,0){34}}
      \put(-16,15){\HBCenter{\scriptsize $2$}}
      \put(16,15){\HBCenter{\scriptsize $2$}}
      \put(0,-7){\HVCenter{\scriptsize $1$}}
    }
  }
}
\put(10,65){
  \put(20,0){
    \put(-42,30){\HVCenter{(c)}}
    \put(0,0){
      \put(-20,0){\circle{4}}
      \put(20,0){\circle{4}}
      \put(0,30){\circle*{4}}
      \put(-28,0){\HVCenter{$i$}}
      \put(28,0){\HVCenter{$j$}}
      \put(0,34){\HBCenter{$k$}}
      \put(-18.2,2.7){\vector(2,3){16.4}}
      \put(18.2,2.7){\vector(-2,3){16.4}}
      \put(-16,15){\HBCenter{\scriptsize $1$}}
      \put(16,15){\HBCenter{\scriptsize $2$}}
    }
    \put(45,16){
      \put(0,0){\HVCenter{$\longleftrightarrow$}}
      \put(0,6){\HBCenter{$\mu_k$}}
    }
    \put(90,0){
      \put(-20,0){\circle{4}}
      \put(20,0){\circle{4}}
      \put(0,30){\circle*{4}}
      \put(-28,0){\HVCenter{$i$}}
      \put(28,0){\HVCenter{$j$}}
      \put(0,34){\HBCenter{$k$}}
      \put(-1.8,27.3){\vector(-2,-3){16.4}}
      \put(1.8,27.3){\vector(2,-3){16.4}}
      \put(-16,15){\HBCenter{\scriptsize $1$}}
      \put(16,15){\HBCenter{\scriptsize $2$}}
    }
  }
  \put(240,0){
    \put(-42,30){\HVCenter{(d)}}
    \put(0,0){
      \put(-20,0){\circle{4}}
      \put(20,0){\circle{4}}
      \put(0,30){\circle*{4}}
      \put(-28,0){\HVCenter{$i$}}
      \put(28,0){\HVCenter{$j$}}
      \put(0,34){\HBCenter{$k$}}
      \put(-18.2,2.7){\vector(2,3){16.4}}
      \put(1.8,27.3){\vector(2,-3){16.4}}
      \put(-16,15){\HBCenter{\scriptsize $2$}}
      \put(16,15){\HBCenter{\scriptsize $1$}}
    }
    \put(45,16){
      \put(0,0){\HVCenter{$\longleftrightarrow$}}
      \put(0,6){\HBCenter{$\mu_k$}}
    }
    \put(90,0){
      \put(-20,0){\circle{4}}
      \put(20,0){\circle{4}}
      \put(0,30){\circle*{4}}
      \put(-28,0){\HVCenter{$i$}}
      \put(28,0){\HVCenter{$j$}}
      \put(0,34){\HBCenter{$k$}}
      \put(18.2,2.7){\vector(-2,3){16.4}}
      \put(-1.8,27.3){\vector(-2,-3){16.4}}
      \put(-15,0){\vector(1,0){30}}
      \put(-16,15){\HBCenter{\scriptsize $2$}}
      \put(16,15){\HBCenter{\scriptsize $1$}}
      \put(0,-7){\HVCenter{\scriptsize $2$}}
    }
  }
}
\put(10,130){
  \put(20,0){
    \put(-42,30){\HVCenter{(a)}}
    \put(0,0){
      \put(-20,0){\circle{4}}
      \put(20,0){\circle{4}}
      \put(0,30){\circle*{4}}
      \put(-28,0){\HVCenter{$i$}}
      \put(28,0){\HVCenter{$j$}}
      \put(0,34){\HBCenter{$k$}}
      \put(-18.2,2.7){\vector(2,3){16.4}}
      \put(18.2,2.7){\vector(-2,3){16.4}}
      \put(-16,15){\HBCenter{\scriptsize $1$}}
      \put(16,15){\HBCenter{\scriptsize $1$}}
    }
    \put(45,16){
      \put(0,0){\HVCenter{$\longleftrightarrow$}}
      \put(0,6){\HBCenter{$\mu_k$}}
    }
    \put(90,0){
      \put(-20,0){\circle{4}}
      \put(20,0){\circle{4}}
      \put(0,30){\circle*{4}}
      \put(-28,0){\HVCenter{$i$}}
      \put(28,0){\HVCenter{$j$}}
      \put(0,34){\HBCenter{$k$}}
      \put(-1.8,27.3){\vector(-2,-3){16.4}}
      \put(1.8,27.3){\vector(2,-3){16.4}}
      \put(-16,15){\HBCenter{\scriptsize $1$}}
      \put(16,15){\HBCenter{\scriptsize $1$}}
    }
  }
  \put(240,0){
    \put(-42,30){\HVCenter{(b)}}
    \put(0,0){
      \put(-20,0){\circle{4}}
      \put(20,0){\circle{4}}
      \put(0,30){\circle*{4}}
      \put(-28,0){\HVCenter{$i$}}
      \put(28,0){\HVCenter{$j$}}
      \put(0,34){\HBCenter{$k$}}
      \put(-18.2,2.7){\vector(2,3){16.4}}
      \put(1.8,27.3){\vector(2,-3){16.4}}
      \put(-16,15){\HBCenter{\scriptsize $1$}}
      \put(16,15){\HBCenter{\scriptsize $1$}}
    }
    \put(45,16){
      \put(0,0){\HVCenter{$\longleftrightarrow$}}
      \put(0,6){\HBCenter{$\mu_k$}}
    }
    \put(90,0){
      \put(-20,0){\circle{4}}
      \put(20,0){\circle{4}}
      \put(0,30){\circle*{4}}
      \put(-28,0){\HVCenter{$i$}}
      \put(28,0){\HVCenter{$j$}}
      \put(0,34){\HBCenter{$k$}}
      \put(18.2,2.7){\vector(-2,3){16.4}}
      \put(-1.8,27.3){\vector(-2,-3){16.4}}
      \put(-15,0){\vector(1,0){30}}
      \put(-16,15){\HBCenter{\scriptsize $1$}}
      \put(16,15){\HBCenter{\scriptsize $1$}}
      \put(0,-7){\HVCenter{\scriptsize $1$}}
    }
  }
}
\end{picture}
\caption{Local picture of mutation of a diagram of finite type.}
\label{f:localmutation}
\end{figure}
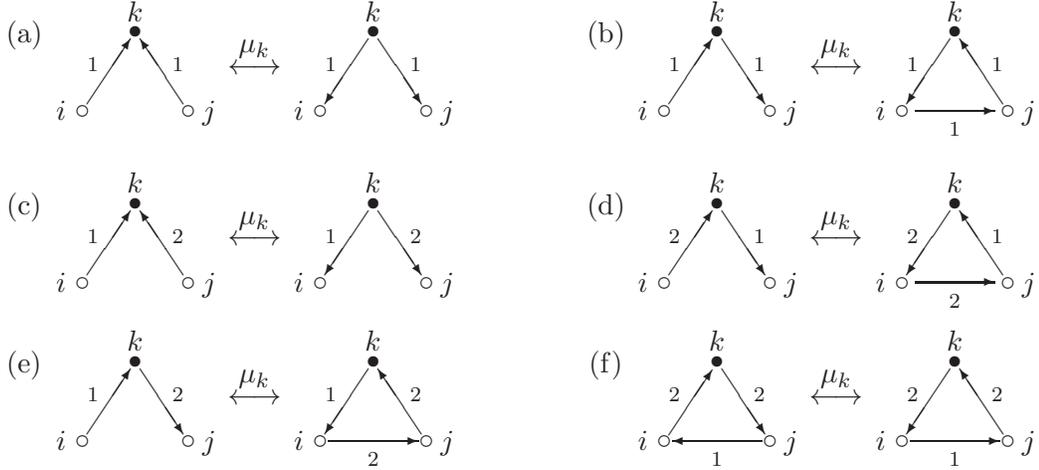

We say that a vertex in a diagram $\Gamma$
is \emph{connected} to another if there is an edge between them.
The next result describes the way vertices in $\Gamma$ can be connected
to a chordless cycle.

\begin{lemma} \label{l:keylemma}
Suppose that a vertex $k$ is connected to at least one vertex in a
chordless cycle in a diagram $\Gamma$ of finite type.
Then $k$ is connected to at most two vertices of the cycle.
If it is connected to two vertices, they must be adjacent in the cycle.
\end{lemma}

\begin{proof}
Consider a chordless cycle $C$:
$$i_0\rightarrow i_1\rightarrow \cdots \rightarrow i_{r-1}\rightarrow i_0$$
in $\Gamma$, and suppose that $k$ is connected to
$i_{a_1}$, $i_{a_2},\ldots $, $i_{a_t}$ for some $t\leq r$, where
$0\leq a_1<a_2<\cdots <a_t\leq r-1$. Suppose first that $t>2$. Since, by
Proposition~\ref{p:cycles}, the chordless cycles (in the underlying unoriented
graph)
$$\xymatrix{
i_{a_1} \ar@{-}[r] & i_{a_1+1} \ar@{-}[r] & \cdots \ar@{-} & i_{a_2} \ar@{-}[r]
& k \ar@{-}[r] & i_{a_1}
}$$
and
$$
\xymatrix{
i_{a_2} \ar@{-}[r] & i_{a_2+1} \ar@{-}[r] & \cdots \ar@{-} & i_{a_3} \ar@{-}[r]
& k \ar@{-}[r] & i_{a_2}
}$$
are cyclically oriented in $\Gamma$, we must have
$$i_{a_2}\rightarrow k \rightarrow i_{a_1}$$
from the first and
$$i_{a_3}\rightarrow k \rightarrow i_{a_2}$$
from the second, giving a contradiction.

If $r=2$, consider the cycle
$$\xymatrix{
i_{a_2} \ar@{-}[r] & i_{a_2+1} \ar@{-}[r] & \cdots \ar@{-} & i_{a_1} \ar@{-}[r]
& k \ar@{-}[r] & i_{a_2}
}$$
(where the subscripts are written mod $r$). If it is a chordless
cycle, it must be oriented, giving $$i_{a_1}\rightarrow k\rightarrow i_{a_2},$$
again giving a contradiction. Hence this last cycle is not a chordless
cycle, which implies that $i_{a_1}$ and $i_{a_2}$ are adjacent in $C$
as required.
\end{proof}

\begin{lemma} \label{l:newcycles}
Let $\Gamma$ be a diagram of finite type and $\Gamma'=\mu_k(\Gamma)$
the mutation of $\Gamma$ at vertex $k$.
In Figures~\ref{f:mutatecycle} and~\ref{f:mutatecycle2}
we list various types of induced
subdiagrams in $\Gamma$ (on the left) and corresponding cycles $C'$ in
$\Gamma'$ (on the right) arising from the mutation of $\Gamma$ at $k$.
The diagrams are drawn so that $C'$ is always a clockwise cycle, and in
case $(i)$, $C'$ has at least $3$ vertices, while in case (j),
$C'$ has at least $4$ vertices.
Then every chordless cycle in $\Gamma'$ arises in such a way.
\end{lemma}

\begin{figure}[!ht]
\begin{picture}(350,255)
\put(10,5){
  \put(18,0){
    \put(-42,40){\HVCenter{(g)}}
    \put(0,18){
      \put(-18,-18){\circle{4}}
      \put(18,-18){\circle{4}}
      \put(-18,18){\circle*{4}}
      \put(-24,24){\HVCenter{$k$}}
      \put(18,18){\circle{4}}
      \put(-15,18){\vector(1,0){30}}
      \put(-18,-15){\vector(0,1){30}}
      \put(18,15){\vector(0,-1){30}}
      \put(15,-18){\vector(-1,0){30}}
      \put(-23,0){\HVCenter{\scriptsize $2$}}
      \put(23,0){\HVCenter{\scriptsize $2$}}
      \put(0,21){\HBCenter{\scriptsize $1$}}
      \put(0,-21){\HTCenter{\scriptsize $1$}}
    }
    \put(45,18){
      \put(0,0){\HVCenter{$\longrightarrow$}}
      \put(0,6){\HBCenter{$\mu_k$}}
    }
    \put(90,18){
      \put(-18,-18){\circle{4}}
      \put(18,-18){\circle{4}}
      \put(-18,18){\circle*{4}}
      \put(-24,24){\HVCenter{$k$}}
      \put(18,18){\circle{4}}
      \put(15,18){\vector(-1,0){30}}
      \put(-18,15){\vector(0,-1){30}}
      \put(18,15){\vector(0,-1){30}}
      \put(15,-18){\vector(-1,0){30}}
      \put(-18,-18){\put(2,2){\vector(1,1){32}}}
      \put(-23,0){\HVCenter{\scriptsize $2$}}
      \put(23,0){\HVCenter{\scriptsize $2$}}
      \put(0,21){\HBCenter{\scriptsize $1$}}
      \put(0,-21){\HTCenter{\scriptsize $1$}}
      \put(-4,4){\HVCenter{\scriptsize $2$}}
      \put(8,-8){\HVCenter{$C'$}}
    }
  }
  \put(240,0){
    \put(-42,36){\HVCenter{(h)}}
    \put(0,18){
      \put(-18,-18){\circle{4}}
      \put(18,-18){\circle{4}}
      \put(18,18){\circle*{4}}
      \put(24,24){\HVCenter{$k$}}
      \put(-18,18){\circle{4}}
      \put(-15,18){\vector(1,0){30}}
      \put(-18,-15){\vector(0,1){30}}
      \put(18,15){\vector(0,-1){30}}
      \put(15,-18){\vector(-1,0){30}}
      \put(-23,0){\HVCenter{\scriptsize $2$}}
      \put(23,0){\HVCenter{\scriptsize $2$}}
      \put(0,21){\HBCenter{\scriptsize $1$}}
      \put(0,-21){\HTCenter{\scriptsize $1$}}
    }
    \put(45,18){
      \put(0,0){\HVCenter{$\longrightarrow$}}
      \put(0,6){\HBCenter{$\mu_k$}}
    }
    \put(90,18){
      \put(-18,-18){\circle{4}}
      \put(18,-18){\circle{4}}
      \put(18,18){\circle*{4}}
      \put(24,24){\HVCenter{$k$}}
      \put(-18,18){\circle{4}}
      \put(15,18){\vector(-1,0){30}}
      \put(-18,-15){\vector(0,1){30}}
      \put(18,-15){\vector(0,1){30}}
      \put(15,-18){\vector(-1,0){30}}
      \put(-18,18){\put(2,-2){\vector(1,-1){32}}}
      \put(-23,0){\HVCenter{\scriptsize $2$}}
      \put(23,0){\HVCenter{\scriptsize $2$}}
      \put(0,21){\HBCenter{\scriptsize $1$}}
      \put(0,-21){\HTCenter{\scriptsize $1$}}
      \put(4,4){\HVCenter{\scriptsize $2$}}
      \put(-8,-8){\HVCenter{$C'$}}
    }
  }
}
\put(10,75){
  \put(18,0){
    \put(-42,40){\HVCenter{(e)}}
    \put(90,18){
      \put(-18,-18){\circle{4}}
      \put(18,-18){\circle{4}}
      \put(-18,18){\circle*{4}}
      \put(-24,24){\HVCenter{$k$}}
      \put(18,18){\circle{4}}
      \put(-15,18){\vector(1,0){30}}
      \put(-18,-15){\vector(0,1){30}}
      \put(18,15){\vector(0,-1){30}}
      \put(15,-18){\vector(-1,0){30}}
      \put(-23,0){\HVCenter{\scriptsize $2$}}
      \put(23,0){\HVCenter{\scriptsize $2$}}
      \put(0,21){\HBCenter{\scriptsize $1$}}
      \put(0,-21){\HTCenter{\scriptsize $1$}}
      \put(0,0){\HVCenter{$C'$}}
    }
    \put(45,18){
      \put(0,0){\HVCenter{$\longrightarrow$}}
      \put(0,6){\HBCenter{$\mu_k$}}
    }
    \put(0,18){
      \put(-18,-18){\circle{4}}
      \put(18,-18){\circle{4}}
      \put(-18,18){\circle*{4}}
      \put(-24,24){\HVCenter{$k$}}
      \put(18,18){\circle{4}}
      \put(15,18){\vector(-1,0){30}}
      \put(-18,15){\vector(0,-1){30}}
      \put(18,15){\vector(0,-1){30}}
      \put(15,-18){\vector(-1,0){30}}
      \put(-18,-18){\put(2,2){\vector(1,1){32}}}
      \put(-23,0){\HVCenter{\scriptsize $2$}}
      \put(23,0){\HVCenter{\scriptsize $2$}}
      \put(0,21){\HBCenter{\scriptsize $1$}}
      \put(0,-21){\HTCenter{\scriptsize $1$}}
      \put(-4,4){\HVCenter{\scriptsize $2$}}
    }
  }
  \put(240,0){
    \put(-42,36){\HVCenter{(f)}}
    \put(90,18){
      \put(-18,-18){\circle{4}}
      \put(18,-18){\circle{4}}
      \put(18,18){\circle*{4}}
      \put(24,24){\HVCenter{$k$}}
      \put(-18,18){\circle{4}}
      \put(-15,18){\vector(1,0){30}}
      \put(-18,-15){\vector(0,1){30}}
      \put(18,15){\vector(0,-1){30}}
      \put(15,-18){\vector(-1,0){30}}
      \put(-23,0){\HVCenter{\scriptsize $2$}}
      \put(23,0){\HVCenter{\scriptsize $2$}}
      \put(0,21){\HBCenter{\scriptsize $1$}}
      \put(0,-21){\HTCenter{\scriptsize $1$}}
      \put(0,0){\HVCenter{$C'$}}
    }
    \put(45,18){
      \put(0,0){\HVCenter{$\longrightarrow$}}
      \put(0,6){\HBCenter{$\mu_k$}}
    }
    \put(0,18){
      \put(-18,-18){\circle{4}}
      \put(18,-18){\circle{4}}
      \put(18,18){\circle*{4}}
      \put(24,24){\HVCenter{$k$}}
      \put(-18,18){\circle{4}}
      \put(15,18){\vector(-1,0){30}}
      \put(-18,-15){\vector(0,1){30}}
      \put(18,-15){\vector(0,1){30}}
      \put(15,-18){\vector(-1,0){30}}
      \put(-18,18){\put(2,-2){\vector(1,-1){32}}}
      \put(-23,0){\HVCenter{\scriptsize $2$}}
      \put(23,0){\HVCenter{\scriptsize $2$}}
      \put(0,21){\HBCenter{\scriptsize $1$}}
      \put(0,-21){\HTCenter{\scriptsize $1$}}
      \put(4,4){\HVCenter{\scriptsize $2$}}
    }
  }
}
\put(10,145){
  \put(20,0){
    \put(-42,30){\HVCenter{(c)}}
    \put(0,0){
      \put(-20,0){\circle{4}}
      \put(20,0){\circle{4}}
      \put(0,30){\circle*{4}}
      \put(0,35){\HBCenter{$k$}}
      \put(-1.8,27.3){\vector(-2,-3){16.4}}
      \put(18.2,2.7){\vector(-2,3){16.4}}
      \put(-16,15){\HBCenter{\scriptsize $1$}}
      \put(16,15){\HBCenter{\scriptsize $2$}}
    }
    \put(45,16){
      \put(0,0){\HVCenter{$\longrightarrow$}}
      \put(0,6){\HBCenter{$\mu_k$}}
    }
    \put(90,0){
      \put(-20,0){\circle{4}}
      \put(20,0){\circle{4}}
      \put(0,30){\circle*{4}}
      \put(0,35){\HBCenter{$k$}}
      \put(-18.2,2.7){\vector(2,3){16.4}}
      \put(1.4,27.3){\vector(2,-3){16.4}}
      \put(17,0){\vector(-1,0){34}}
      \put(-16,15){\HBCenter{\scriptsize $1$}}
      \put(16,15){\HBCenter{\scriptsize $2$}}
      \put(0,-7){\HVCenter{\scriptsize $2$}}
      \put(0,10){\HVCenter{$C'$}}
    }
  }
  \put(240,0){
    \put(-42,30){\HVCenter{(d)}}
    \put(0,0){
      \put(-20,0){\circle{4}}
      \put(20,0){\circle{4}}
      \put(0,30){\circle*{4}}
      \put(0,35){\HBCenter{$k$}}
      \put(-1.8,27.3){\vector(-2,-3){16.4}}
      \put(18.2,2.7){\vector(-2,3){16.4}}
      \put(-17,0){\vector(1,0){34}}
      \put(-16,15){\HBCenter{\scriptsize $2$}}
      \put(16,15){\HBCenter{\scriptsize $2$}}
      \put(0,-7){\HVCenter{\scriptsize $1$}}
    }
    \put(45,16){
      \put(0,0){\HVCenter{$\longrightarrow$}}
      \put(0,6){\HBCenter{$\mu_k$}}
    }
    \put(90,0){
      \put(-20,0){\circle{4}}
      \put(20,0){\circle{4}}
      \put(0,30){\circle*{4}}
      \put(0,35){\HBCenter{$k$}}
      \put(-18.2,2.7){\vector(2,3){16.4}}
      \put(1.8,27.3){\vector(2,-3){16.4}}
      \put(17,0){\vector(-1,0){34}}
      \put(-16,15){\HBCenter{\scriptsize $2$}}
      \put(16,15){\HBCenter{\scriptsize $2$}}
      \put(0,-7){\HVCenter{\scriptsize $1$}}
      \put(0,10){\HVCenter{$C'$}}
    }
  }
}
\put(10,210){
  \put(20,0){
    \put(-42,30){\HVCenter{(a)}}
    \put(0,0){
      \put(-20,0){\circle{4}}
      \put(20,0){\circle{4}}
      \put(0,30){\circle*{4}}
      \put(0,35){\HBCenter{$k$}}
      \put(-1.8,27.3){\vector(-2,-3){16.4}}
      \put(18.2,2.7){\vector(-2,3){16.4}}
      \put(-16,15){\HBCenter{\scriptsize $1$}}
      \put(16,15){\HBCenter{\scriptsize $1$}}
    }
    \put(45,16){
      \put(0,0){\HVCenter{$\longrightarrow$}}
      \put(0,6){\HBCenter{$\mu_k$}}
    }
    \put(90,0){
      \put(-20,0){\circle{4}}
      \put(20,0){\circle{4}}
      \put(0,30){\circle*{4}}
      \put(0,35){\HBCenter{$k$}}
      \put(-18.2,2.7){\vector(2,3){16.4}}
      \put(1.4,27.3){\vector(2,-3){16.4}}
      \put(17,0){\vector(-1,0){34}}
      \put(-16,15){\HBCenter{\scriptsize $1$}}
      \put(16,15){\HBCenter{\scriptsize $1$}}
      \put(0,-7){\HVCenter{\scriptsize $1$}}
      \put(0,10){\HVCenter{$C'$}}
    }
  }
  \put(240,0){
    \put(-42,30){\HVCenter{(b)}}
    \put(0,0){
      \put(-20,0){\circle{4}}
      \put(20,0){\circle{4}}
      \put(0,30){\circle*{4}}
      \put(0,35){\HBCenter{$k$}}
      \put(-1.8,27.3){\vector(-2,-3){16.4}}
      \put(18.2,2.7){\vector(-2,3){16.4}}
      \put(-16,15){\HBCenter{\scriptsize $2$}}
      \put(16,15){\HBCenter{\scriptsize $1$}}
    }
    \put(45,16){
      \put(0,0){\HVCenter{$\longrightarrow$}}
      \put(0,6){\HBCenter{$\mu_k$}}
    }
    \put(90,0){
      \put(-20,0){\circle{4}}
      \put(20,0){\circle{4}}
      \put(0,30){\circle*{4}}
      \put(0,35){\HBCenter{$k$}}
      \put(-18.2,2.7){\vector(2,3){16.4}}
      \put(1.4,27.3){\vector(2,-3){16.4}}
      \put(17,0){\vector(-1,0){34}}
      \put(-16,15){\HBCenter{\scriptsize $2$}}
      \put(16,15){\HBCenter{\scriptsize $1$}}
      \put(0,-7){\HVCenter{\scriptsize $2$}}
      \put(0,10){\HVCenter{$C'$}}
    }
  }
}
\end{picture}
\caption{Induced subdiagrams of $\Gamma$ and the corresponding chordless cycles
in $\Gamma'=\mu_k(\Gamma)$, part 1.}
\label{f:mutatecycle}
\end{figure}
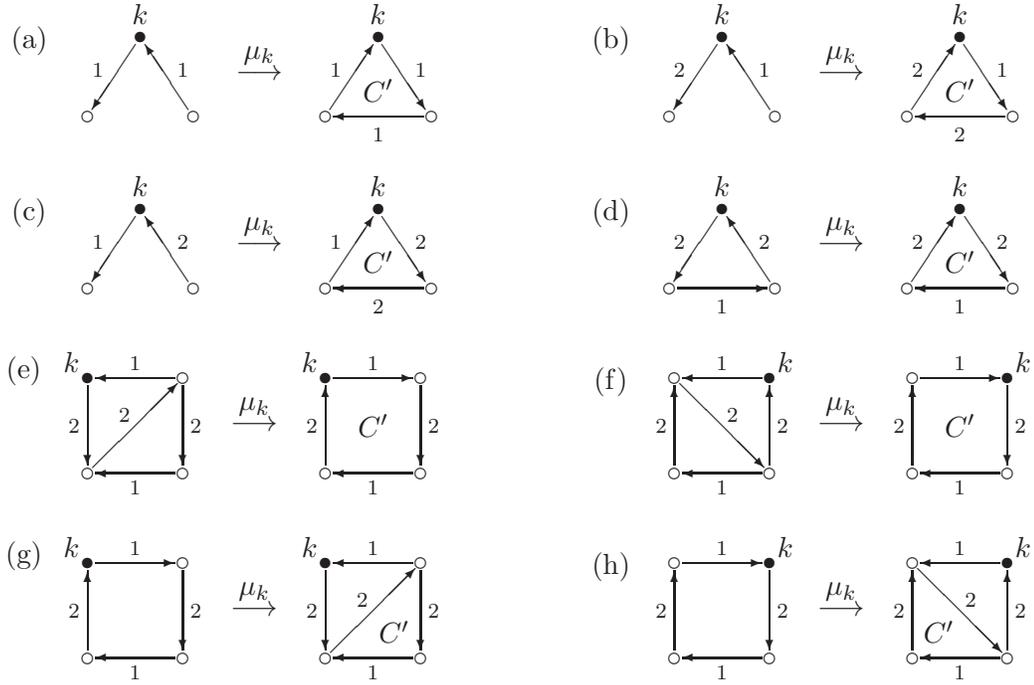

\begin{figure}[!ht]
\begin{center}
\begin{picture}(346,140)
\put(56,90){
  \put(-56,40){\HVCenter{(i)}}
  \put(0,20){
    \multiput(0,-20)(0,40){2}{\circle{4}}
    \multiput(40,-20)(0,40){2}{\circle{4}}
    \multiput(80,-20)(0,40){2}{\circle{4}}
    \put(-30,0){\circle*{4}}
    \put(-38,0){\HVCenter{\small $k$}}
    \multiput(0,20)(40,0){2}{\put(3,0){\vector(1,0){34}}}
    \multiput(40,-20)(40,0){2}{\put(-3,0){\vector(-1,0){34}}}
    \put(0,-20){\put(-2.4,1.8){\vector(-3,2){26.2}}}
    \put(-30,0){\put(2.4,1.8){\vector(3,2){26.2}}}
    \put(80,20){\qbezier(3,0)(15,0)(15,-15)}
    \put(80,-20){\qbezier(3,0)(15,0)(15,15)\put(5,0){\vector(-1,0){2}}}
    \multiput(95,-3.5)(0,3){3}{\line(0,1){1}}
    \put(-19,14.3){\HVCenter{\scriptsize $1$}}
    \put(-19,-14.3){\HVCenter{\scriptsize $1$}}
    \multiput(20,26.3)(40,0){2}{\HVCenter{\scriptsize $1$}}
    \multiput(20,-26.3)(40,0){2}{\HVCenter{\scriptsize $1$}}
  }
  \put(130,20){
    \put(0,0){\HVCenter{$\longrightarrow$}}
    \put(0,7){\HVCenter{$\mu_k$}}
  }
  \put(195,20){
    \multiput(0,-20)(0,40){2}{\circle{4}}
    \multiput(40,-20)(0,40){2}{\circle{4}}
    \multiput(80,-20)(0,40){2}{\circle{4}}
    \put(-30,0){\circle*{4}}
    \put(-38,0){\HVCenter{\small $k$}}
    \multiput(0,20)(40,0){2}{\put(3,0){\vector(1,0){34}}}
    \multiput(40,-20)(40,0){2}{\put(-3,0){\vector(-1,0){34}}}
    \put(0,-17){\vector(0,1){34}}
    \put(0,20){\put(-2.4,-1.8){\vector(-3,-2){26.2}}}
    \put(-30,0){\put(2.4,-1.8){\vector(3,-2){26.2}}}
    \put(80,20){\qbezier(3,0)(15,0)(15,-15)}
    \put(80,-20){\qbezier(3,0)(15,0)(15,15)\put(5,0){\vector(-1,0){2}}}
    \multiput(95,-3.5)(0,3){3}{\line(0,1){1}}
    \put(-19,14.3){\HVCenter{\scriptsize $1$}}
    \put(-19,-14.3){\HVCenter{\scriptsize $1$}}
    \multiput(20,26.3)(40,0){2}{\HVCenter{\scriptsize $1$}}
    \multiput(20,-26.3)(40,0){2}{\HVCenter{\scriptsize $1$}}
    \put(5.8,0){\HVCenter{\scriptsize $1$}}
    \put(40,0){\HVCenter{$C'$}}
  }
}
\put(56,15){
  \put(-56,40){\HVCenter{(j)}}
  \put(195,20){
    \multiput(0,-20)(0,40){2}{\circle{4}}
    \multiput(40,-20)(0,40){2}{\circle{4}}
    \multiput(80,-20)(0,40){2}{\circle{4}}
    \put(-30,0){\circle*{4}}
    \put(-38,0){\HVCenter{\small $k$}}
    \multiput(0,20)(40,0){2}{\put(3,0){\vector(1,0){34}}}
    \multiput(40,-20)(40,0){2}{\put(-3,0){\vector(-1,0){34}}}
    \put(0,-20){\put(-2.4,1.8){\vector(-3,2){26.2}}}
    \put(-30,0){\put(2.4,1.8){\vector(3,2){26.2}}}
    \put(80,20){\qbezier(3,0)(15,0)(15,-15)}
    \put(80,-20){\qbezier(3,0)(15,0)(15,15)\put(5,0){\vector(-1,0){2}}}
    \multiput(95,-3.5)(0,3){3}{\line(0,1){1}}
    \put(-19,14.3){\HVCenter{\scriptsize $1$}}
    \put(-19,-14.3){\HVCenter{\scriptsize $1$}}
    \multiput(20,26.3)(40,0){2}{\HVCenter{\scriptsize $1$}}
    \multiput(20,-26.3)(40,0){2}{\HVCenter{\scriptsize $1$}}
    \put(40,0){\HVCenter{$C'$}}
  }
  \put(130,20){
    \put(0,0){\HVCenter{$\longrightarrow$}}
    \put(0,7){\HVCenter{$\mu_k$}}
  }
  \put(0,20){
    \multiput(0,-20)(0,40){2}{\circle{4}}
    \multiput(40,-20)(0,40){2}{\circle{4}}
    \multiput(80,-20)(0,40){2}{\circle{4}}
    \put(-30,0){\circle*{4}}
    \put(-38,0){\HVCenter{\small $k$}}
    \multiput(0,20)(40,0){2}{\put(3,0){\vector(1,0){34}}}
    \multiput(40,-20)(40,0){2}{\put(-3,0){\vector(-1,0){34}}}
    \put(0,-17){\vector(0,1){34}}
    \put(0,20){\put(-2.4,-1.8){\vector(-3,-2){26.2}}}
    \put(-30,0){\put(2.4,-1.8){\vector(3,-2){26.2}}}
    \put(80,20){\qbezier(3,0)(15,0)(15,-15)}
    \put(80,-20){\qbezier(3,0)(15,0)(15,15)\put(5,0){\vector(-1,0){2}}}
    \multiput(95,-3.5)(0,3){3}{\line(0,1){1}}
    \put(-19,14.3){\HVCenter{\scriptsize $1$}}
    \put(-19,-14.3){\HVCenter{\scriptsize $1$}}
    \multiput(20,26.3)(40,0){2}{\HVCenter{\scriptsize $1$}}
    \multiput(20,-26.3)(40,0){2}{\HVCenter{\scriptsize $1$}}
    \put(5.8,0){\HVCenter{\scriptsize $1$}}
  }
}
\end{picture}
\vskip5pt
\begin{picture}(346,65)
\put(30,0){
  \put(0,60){
    \put(-30,0){\HTCenter{(k)}}
    \put(0,0){\LTCenter{\parbox[l]{314pt}{$C$ is an oriented
          cycle in $\Gamma$ not connected to $k$.
          Then $C'$ is the corresponding cycle in $\Gamma'$.}}}
  }
  \put(0,25){
    \put(-30,0){\HTCenter{(l)}}
    \put(0,0){\LTCenter{\parbox[l]{314pt}{$C$ is an oriented cycle
            in $\Gamma$ with exactly one vertex connected to $k$ (via an edge of unspecified weight).
            Then $C'$ is the corresponding cycle in $\Gamma'$.}}}
  }
}
\end{picture}
\end{center}
\caption{Induced subdiagrams of $\Gamma$ and the corresponding chordless cycles
in $\Gamma'=\mu_k(\Gamma)$, part 2.}
\label{f:mutatecycle2}
\end{figure}
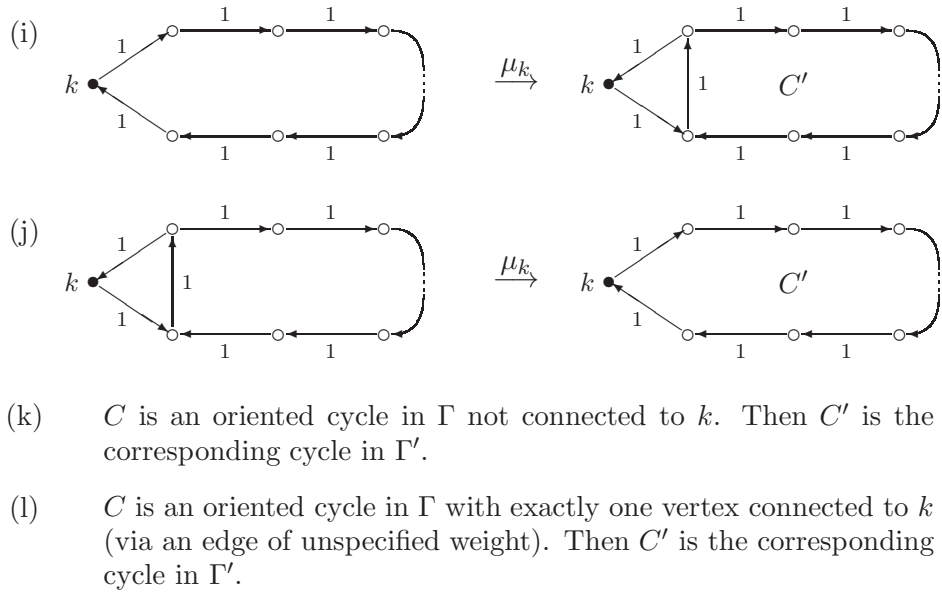

\begin{proof}
Let $C'$ be a chordless cycle in $\Gamma'$.
We consider the different possibilities for the vertex $k$ in relation to $C'$.
If $k$ is a vertex of $C'$, then it follows from Proposition~\ref{p:cycles}
that $C'$ arises as in case (a)-(f) or (j).
If $k$ does not lie in $C'$ and is not connected to any vertex of $C'$
then $C'$ arises as in case (k).
If $k$ does not lie in $C'$ and is connected to exactly one vertex of $C'$,
then $C'$ arises as in case (l). Note that the edge linking $k$ to $C'$
must have weight $1$ or $2$ (since a diagram with two edges with weights
$3$ and $a=1,2$ or $3$ is not of finite type; see the proof
of~\cite[Prop.\ 9.3]{fominzelevinsky2}).

By Lemma~\ref{l:keylemma}, the only remaining case is where $k$ does not
lie in $C'$ and $k$ is connected to two adjacent vertices of $C'$.
If the $3$-cycle $C_k$ containing these $3$ vertices in $\Gamma'$
has all of its weights equal to $1$, then $C'$ also has this
property, since the vertices in $C_k\cup C'$ form a chordless cycle
in $\Gamma$ containing two consecutive edges of weight $1$.
Then $C'$ is as in case (i).

If $C_k$ has weights $1,2,2$ then it is not possible for the edge of
weight $1$ to be shared with $C'$, since then in $\Gamma$,
the vertices of $C'$ form a chordless cycle which is not oriented.
If the common edge has weight $2$ and $C'$ has more than $3$ vertices,
then the vertices of $C_k\cup C'$ form a chordless cycle in $\Gamma$
with at least $5$ vertices and at least one edge of weight $2$, contradicting
Proposition~\ref{p:cycles}. The only other possibility is when
$C'$ has exactly $3$ vertices, and then, by Proposition~\ref{p:cycles},
it must arise as in (g) or (h). The result is proved.
\end{proof}

\section{The group of a diagram arising in a cluster algebra of finite type}
\label{s:groupdefinition}

In this section we define a group (by generators and relations) for any
diagram $\Gamma$ of finite type. We shall see later that it is isomorphic to the
reflection group of the same Dynkin type as the cluster algebra in
which $\Gamma$ arises.

Let $\Gamma$ be a diagram of finite type. For vertices $i,j$ of $\Gamma$
define
$$m_{ij}=
\begin{cases}
2 & \text{if $i$ and $j$ are not connected;} \\
3 & \text{if $i$ and $j$ are connected by an edge of weight $1$;} \\
4 & \text{if $i$ and $j$ are connected by an edge of weight $2$;} \\
6 & \text{if $i$ and $j$ are connected by an edge of weight $3$.}
\end{cases}
$$
Then we define $W_{\Gamma}$ to be the group with generators $s_i$,
$i=1,2,\ldots ,n$, subject to the following relations:
\begin{enumerate}
\item[(R1)] $s_i^2=e$ for all $i$,
\item[(R2)] $(s_is_j)^{m_{ij}}=e$ for all $i\not=j$,
\end{enumerate}
where $e$ denotes the identity element of $W_\Gamma$.

Note that, in the presence of (R1), relation (R2) is symmetric
in that $(s_js_i)^{m_{ji}}=e$ follows from
$(s_is_j)^{m_{ij}}=e$, since $m_{ij}$ is symmetric in $i$ and $j$.
We shall usually use this fact without comment.

For the remaining relations, we suppose that $C$, given by:
$$i_0\rightarrow i_1\rightarrow \cdots \rightarrow i_{d-1}\rightarrow i_0$$
is a chordless cycle in $\Gamma$. For $a=0,1,2,\ldots ,d-1$, define
(with subscripts modulo $d$):
$$r(i_a,i_{a+1})=s_{i_a}s_{i_{a+1}}\cdots s_{i_{a+d-1}}s_{i_{a+d-2}}\cdots s_{i_{a+1}}.$$

\begin{enumerate}
\item[(R3)(a)] If all of the weights in the edges of $C$ are equal to $1$,
then we set $r(i_a,i_{a+1})^2=e$ for $a=0,1,\ldots ,d-1$.
\item[(R3)(b)] If $C$ has some edges of weight $2$, then, for $a=0,1,\ldots ,d-1$, set
$r(i_a,i_{a+1})^k=e$ where $k=4-w_a$ and $w_a$ is the weight of the edge between
$i_a$ and $i_{a-1}$.
\end{enumerate}

Note that if we think of $r(i_a,i_{a+1})$
as a path in $\Gamma$, then in (R3)(b), $w_a$
is the weight of the only edge in $C$ not used in the path.

\begin{remark} \label{r:Dynkincase} \rm
We note that if $\Gamma$ is the diagram associated to a matrix whose
Cartan counterpart is the Cartan matrix of a Dynkin diagram $\Delta$,
then $W_{\Gamma}$ is the reflection group of type $\Delta$.
In this case, the relations (R3)(a) and
(R3)(b) do not occur, since $\Gamma$ is a tree.
\end{remark}

\begin{remark} \label{r:braid} \rm
We also note that, as in Coxeter groups,
if $i\not=j$ then the relation (R2) for the pair $i,j$ is
equivalent (using (R1)) to:
\begin{equation*} \label{e:braid}
s_is_js_i\cdots =s_js_is_j\cdots,
\end{equation*}
where there are $m_{ij}$ terms on each side of the equation.
In particular, if there is no edge in $\Gamma$ between $i$ and $j$,
then $s_is_j=s_is_j$, and if there is an edge in $\Gamma$ between $i$ and
$j$ with weight $1$, then $s_is_js_i=s_js_is_j$. We shall use these
relations liberally.
\end{remark}

\begin{remark} \rm
We shall see later (using the companion basis perspective) that,
in the simply-laced case, the group $W_{\Gamma}$ coincides with a group
defined in~\cite[Defn.\ 6.2]{cst}. See the end of
Section~\ref{s:companionbasis}. In this context, the elements
$r(i_a,i_{a+1})^2$ are referred to as \emph{cut elements}.
\end{remark}

\begin{example} \label{e:d4example}
We consider the example in Figure~\ref{f:d4example}, which is the diagram
of a seed of type $D_4$.

\begin{figure}[!ht]
\begin{center}
\begin{picture}(40,40)
  \put(20,20){
    \multiput(-20,-20)(40,0){2}{\circle{4}}
    \multiput(-20,20)(40,0){2}{\circle{4}}
    \put(-20,-17){\vector(0,1){34}}
    \put(-17,20){\vector(1,0){34}}
    \put(20,17){\vector(0,-1){34}}
    \put(17,-20){\vector(-1,0){34}}
    \put(0,27){\HVCenter{\scriptsize $1$}}
    \put(0,-27){\HVCenter{\scriptsize $1$}}
    \put(27,0){\HVCenter{\scriptsize $1$}}
    \put(-27,0){\HVCenter{\scriptsize $1$}}
    \put(-26,26){\HVCenter{$1$}}
    \put(26,26){\HVCenter{$2$}}
    \put(26,-26){\HVCenter{$3$}}
    \put(-26,-26){\HVCenter{$4$}}
  }
\end{picture}
\end{center}
\caption{A diagram of type $D_4$.}
\label{f:d4example}
\end{figure}
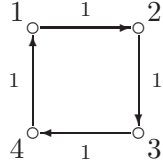

The relations in this case are:
\begin{enumerate}
\item[(R1)] $s_i^2=e$ for $i=1,2,3,4$;
\item[(R2)] $(s_1s_2)^3=(s_2s_3)^3=(s_3s_4)^3=(s_4s_1)^3=e$ and
$(s_1s_3)^2=(s_2s_4)^2=e$;
\item[(R3)] For each element $r(i,j)$ of the following list we have
$r(i,j)^2=e$.
\begin{align*}
r(1,2) &= s_1s_2s_3s_4s_3s_2; \\
r(2,3) &= s_2s_3s_4s_1s_4s_3; \\
r(3,4) &= s_3s_4s_1s_2s_1s_4; \\
r(4,1) &= s_4s_1s_2s_3s_2s_1.
\end{align*}
\end{enumerate}

We shall see in the next section that the relations in (R3)
can be reduced (in fact, it is enough to choose just one of them).
\end{example}

\section{Cyclic Symmetry}
It turns out that, in the presence of the relations (R1) and (R2) in the
above, many of the relations in (R3)(a) and (b) are redundant.
In the sequel, we shall usually use (R1) without comment.

\begin{lemma} \label{l:ncyclesymmetry}
Let $\Gamma$ be a diagram of finite type containing a chordless cycle $C$:
$$i_0\rightarrow i_1\rightarrow \cdots \rightarrow i_{d-1}\rightarrow i_0$$
with the weights of all of its edges equal to $1$.
Let $W$ be the group with generators $s_1,s_2,\ldots ,s_n$ subject
to relations (R1), (R2) and $r(i_a,i_{a+1})^2=e$ for a
single fixed value of $a$. Then all of the relations in (R3)(a) hold for $C$.
\end{lemma}

\begin{proof}
Note first that it is enough to consider the case $a=0$.
For simplicity of notation,
we renumber the vertices of $C$ so that $i_j=j$ for $j=0,1,\ldots ,a-1$.
Thus we assume that $r(0,1)^2=e$ in $W$.

We have, using (R2):
\begin{align*}
r(d-1,0) &= s_{d-1}s_0s_1\cdots s_{d-3}s_{d-2}s_{d-3}\cdots s_1s_0 \\
&= s_0s_0s_{d-1}s_0s_1\cdots s_{d-3}s_{d-2}s_{d-3}\cdots s_1s_{d-1}s_{d-1}s_0 \\
&= s_0s_{d-1}s_0s_{d-1}s_1\cdots s_{d-3}s_{d-2}s_{d-3}\cdots s_1s_{d-1}s_{d-1}s_0 \\
&= s_0s_{d-1}(s_0s_1\cdots s_{d-3}s_{d-1}s_{d-2}s_{d-1}s_{d-3}\cdots s_1)s_{d-1}s_0 \\
&= s_0s_{d-1}(s_0s_1\cdots s_{d-3}s_{d-2}s_{d-1}s_{d-2}s_{d-3}\cdots s_1)s_{d-1}s_0 \\
&= s_0s_{d-1}r(0,1)s_{d-1}s_0.
\end{align*}

It follows that $r(d-1,0)^2=e$.
Repeating this argument inductively, we see that $r(a,a+1)^2=e$ for all $a$,
and we are done.
\end{proof}

\begin{lemma} \label{l:3cyclesymmetry}
Let $\Gamma$ be a diagram of finite type containing a chordless $3$-cycle $C$
as in Figure~\ref{f:3cyclesymmetry} (where we number the vertices $1,2,3$
for convenience).
Let $W$ be the group with generators $s_1,s_2,\ldots ,s_n$ subject
to relations (R1) and (R2) arising from the diagram $\Gamma$.
Then the following are equivalent:
\begin{enumerate}
\item[(a)] $r(1,2)^2=e$;
\item[(b)] $r(2,3)^2=e$.
\end{enumerate}
Furthermore, if one of the above holds, then the following holds:
\begin{enumerate}
\item[(c)] $r(3,1)^3=e$.
\end{enumerate}
\end{lemma}

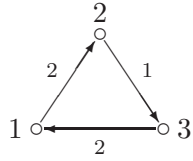
\begin{figure}[!ht]
\begin{center}
\begin{picture}(48,50)
  \put(24,5){
    \multiput(-24,0)(48,0){2}{\circle{4}}
    \put(0,36){\circle{4}}
    \put(21,0){\vector(-1,0){42}}
    \put(1.8,33.3){\vector(2,-3){20.4}}
    \put(-22.4,2.5){\vector(2,3){20.4}}
    \put(0,-7){\HVCenter{\scriptsize $2$}}
    \put(18,22){\HVCenter{\scriptsize $1$}}
    \put(-18,22){\HVCenter{\scriptsize $2$}}
    \put(-32,0){\HVCenter{$1$}}
    \put(32,0){\HVCenter{$3$}}
    \put(0,44){\HVCenter{$2$}}
  }
\end{picture}
\end{center}
\caption{A $3$-cycle (see Lemma~\ref{l:3cyclesymmetry}).}
\label{f:3cyclesymmetry}
\end{figure}

\begin{proof}
Note that $r(1,2)=s_1s_2s_3s_2$, $r(2,3)=s_2s_3s_1s_3$ and
$r(3,1)=s_3s_1s_2s_1$.  \\
The equivalence of (a) and (b) follows from:
$$s_1s_3r(2,3)^2s_3s_1=s_1s_3s_2s_3s_1s_3s_2s_3=s_1s_2s_3s_2s_1s_2s_3s_2=r(1,2)^2$$
(using (R2)). Suppose that (a) and (b) hold.
Then, by (b), $s_3s_1s_3s_2s_3s_1=s_2s_3$. Hence, using (R2),
$$s_1s_3s_1(s_3s_1s_2s_1)^3s_1s_3s_1=(s_1s_3s_1s_3s_1s_2s_1s_1s_3s_1)^3=(s_3s_1s_3s_2s_3s_1)^3=(s_2s_3)^3=e,$$
and (c) follows.
\end{proof}

Note that it is not claimed that (c) is equivalent to
(a) and (b) (we do not know if this holds, but it seems to be unlikely).

\begin{example}
Lemma~\ref{l:ncyclesymmetry} implies that in Example~\ref{e:d4example},
we can replace the relations in (R3) with any fixed one, e.g.\
$$r(1,2)^2=(s_1s_2s_3s_4s_3s_2)^2=e.$$
\end{example}

\begin{lemma} \label{l:4cyclesymmetry}
Let $\Gamma$ be a diagram of finite type containing a chordless $4$-cycle $C$
as in Figure~\ref{f:4cyclesymmetry} (where we number the vertices $1,2,3,4$
for convenience).
Let $W$ be the group with generators $s_1,s_2,\ldots ,s_n$ subject
to relations (R1) and (R2) arising from the diagram $\Gamma$.
Then the following are equivalent:
\begin{enumerate}
\item[(a)] $r(1,2)^2=e$;
\item[(b)] $r(3,4)^2=e$;
\end{enumerate}
Furthermore, if one of the above holds, then the following both hold:
\begin{enumerate}
\item[(c)] $r(2,3)^3=e$;
\item[(d)] $r(4,1)^3=e$.
\end{enumerate}
\end{lemma}

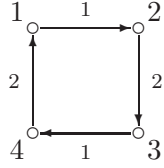
\begin{figure}[!ht]
\begin{center}
\begin{picture}(40,40)
  \put(20,20){
    \multiput(-20,-20)(40,0){2}{\circle{4}}
    \multiput(-20,20)(40,0){2}{\circle{4}}
    \put(-20,-17){\vector(0,1){34}}
    \put(-17,20){\vector(1,0){34}}
    \put(20,17){\vector(0,-1){34}}
    \put(17,-20){\vector(-1,0){34}}
    \put(0,27){\HVCenter{\scriptsize $1$}}
    \put(0,-27){\HVCenter{\scriptsize $1$}}
    \put(27,0){\HVCenter{\scriptsize $2$}}
    \put(-27,0){\HVCenter{\scriptsize $2$}}
    \put(-26,26){\HVCenter{$1$}}
    \put(26,26){\HVCenter{$2$}}
    \put(26,-26){\HVCenter{$3$}}
    \put(-26,-26){\HVCenter{$4$}}
  }
\end{picture}
\end{center}
\caption{A chordless $4$-cycle (see Lemma~\ref{l:4cyclesymmetry}).}
\label{f:4cyclesymmetry}
\end{figure}

\begin{proof}
Note first that $r(1,2)=s_1s_2s_3s_4s_3s_2$, 
$r(2,3)=s_2s_3s_4s_1s_4s_3$, 
$r(3,4)=s_3s_4s_1s_2s_1s_4$, and $r(4,1)=s_4s_1s_2s_3s_2s_1$.

The equivalence of (a) and (b) follows from:
$$s_3s_4s_2r(1,2)s_2s_4s_3=s_3s_4s_2s_1s_2s_3s_4s_3s_4s_3=s_3s_4s_2s_1s_2s_4=s_3s_4s_1s_2s_1s_4=r(3,4),$$
using (R2).
So suppose that (a) and (b) hold.
Applying (R2) to (a), we have:
$$(s_1s_2s_4s_3s_4s_2)^2=e.$$
Thus, noting that $s_2$ and $s_4$ commute, we have:
$$(s_1s_4s_2s_3s_2s_4)^2=e.$$
So we get $(s_4 s_1 s_4 s_2 s_3 s_2)^2=e$ and hence
\begin{equation} \label{e:r12}
s_4s_1s_4s_2s_3s_2=s_2s_3s_2s_4s_1s_4.
\end{equation}
Therefore, using first $s_3s_2s_3=s_2s_3s_2s_3s_2$ and then
equation~\eqref{e:r12}, we get
\begin{align*}
s_3r(2,3)s_3 &= s_3s_2s_3s_4s_1s_4\\
&= s_2s_3(s_2s_3s_2s_4s_1s_4)\\
&= s_2s_3(s_4s_1s_4s_2s_3s_2)\\
&= s_2s_3 s_4 (s_1s_2) s_4 s_3s_2,
\end{align*}
where in the last equation we used that $s_2$ and $s_4$ commute. Consequently
\begin{align*}
(s_3 r(2,3) s_3)^3
&= s_2 s_3 s_4 (s_1 s_2)^3 s_4 s_3 s_2 \\
&=e,
\end{align*}
since $(s_1 s_2)^3=e$. This shows that $r(2,3)^3=e$.

By a similar argument, we see that (b) implies (d), and we are done.
\end{proof}

The following now follows from Proposition~\ref{p:cycles} and
Lemmas~\ref{l:ncyclesymmetry}, \ref{l:3cyclesymmetry}
and~\ref{l:4cyclesymmetry}.

\begin{prop}
Let $\Gamma$ be a diagram of finite type. Then, in the presence of relations (R1) and
(R2), the relations (3) considered in the introduction imply the relations (R3) considered
above. In particular, $W_{\Gamma}$ is isomorphic to the group $W(\Gamma)$ defined in
the introduction.
\end{prop}

Finally, in this section, we consider the behaviour of the relations
under reversing all of the arrows in $\Gamma$.

\begin{prop} \label{p:opposite}
Let $\Gamma$ be a diagram of finite type and let $W$ be a group with
generators $s_1,s_2,\ldots ,s_n$. Then the $s_i$ satisfy relations
(R1), (R2) and (R3) with respect to $\Gamma$ if and only if they satisfy
relations (R1), (R2) and (R3) with respect to $\Gamma^{\op}$.
\end{prop}

\begin{proof}
We assume that the $s_i$ satisfy relations (R1), (R2) and (R3) with respect
to $\Gamma$, and show that they satisfy relations (R1), (R2) and (R3) with
respect to $\Gamma^{\op}$. The converse follows by replacing $\Gamma$ with
$\Gamma^{\op}$. Since (R1) and (R2) do not depend on the orientation of $\Gamma$,
it follows that the $s_i$ satisfy relations (R1) and (R2) with respect to
$\Gamma^{\op}$. In order to show (R3), we note that 
by Lemmas~\ref{l:ncyclesymmetry},~\ref{l:3cyclesymmetry}
and~\ref{l:4cyclesymmetry}, we only have to check one of the relations
of the form $r(i_a,i_a+1)^2=e$ arising from a cycle in order to be sure
that all of the relations corresponding to that cycle hold.

Suppose first that $\Gamma$ contains a chordless cycle $C$:
$$i_0\rightarrow i_1\rightarrow \cdots \rightarrow i_{d-1}\rightarrow i_0$$
with the weights of all of its edges equal to $1$.
We renumber the vertices $i_j$, $j=0,1\ldots ,d-1$ so that they become
$0,1,\ldots ,d-1$. Then, we have:
\begin{align*}
r(d-1,d-2) &= s_{d-1}s_{d-2}\cdots s_1s_0s_1 \cdots s_{d-3}s_{d-2} \\
&= s_{d-1}s_{d-2}\cdots s_1(s_0s_1 \cdots s_{d-3}s_{d-2}s_{d-1}s_{d-2}\cdots s_1)s_1s_2 \cdots s_{d-1} \\
&= s_{d-1}s_{d-2}\cdots s_1 r(0,1) s_1s_2 \cdots s_{d-1}.
\end{align*}
So $r(d-1,d-2)^2=e$.

Suppose next that $\Gamma$ contains a chordless $3$-cycle $C$
as in Figure~\ref{f:3cyclesymmetry} (where we number the vertices
$1,2,3$ for convenience). Then we have
$$s_2s_1r(1,2)s_1s_2=s_2s_1s_1s_2s_3s_2s_1s_2=s_3s_2s_1s_2=r(3,2),$$
so $r(3,2)^2=e$.

Finally, suppose that $\Gamma$ contains a chordless $4$-cycle $C$
as in Figure~\ref{f:4cyclesymmetry} (where we number the vertices $1,2,3,4$
for convenience). Then, using (R2), we have:
\begin{align*}
r(2,1) &= s_2s_1s_4s_3s_4s_1 \\
&= s_1(s_1s_2s_1)(s_4s_3s_4)s_1 \\
&= s_1(s_2s_1s_2)(s_3s_4s_3)
s_2s_2s_1  \\
&= (s_1s_2)(s_1s_2s_3s_4s_3s_2)(s_2s_1) \\
&= (s_1s_2)r(1,2)(s_2s_1);
\end{align*}
hence $r(2,1)^2=e$.

Since every chordless cycle in $\Gamma^{\op}$ corresponds to a chordless
cycle in $\Gamma$, the result follows from Proposition~\ref{p:cycles}.
\end{proof}

\begin{remark}
Proposition~\ref{p:opposite} shows that $W_{\Gamma}$ could be defined
using only the underlying unoriented weighted graph of $\Gamma$, by taking the
relations (R1), (R2) and (R3) corresponding to both $\Gamma$ and
$\Gamma^{\op}$ as the defining relations.
\end{remark}

\section{Invariance of the group under mutation}

In this section we prove the main result, that the group of a diagram
of finite type depends only on the Dynkin type, up to isomorphism. We
will prove this by showing that the group of the diagram is invariant
(up to isomorphism) under mutation.

We fix a diagram $\Gamma$ of finite type and a vertex $k$ of $\Gamma$.
Then we have the mutation $\Gamma'=\mu_k(\Gamma)$.

For a vertex $i$ of $\Gamma$, we define an element $t_i\in W_{\Gamma}$
as follows:
$$t_i=\begin{cases}
s_ks_is_k & \text{if there is an arrow $i\rightarrow k$ in $\Gamma$ (possibly weighted);} \\
s_i & \text{otherwise.}
\end{cases}
$$

Our aim is to show that the elements $t_i$ satisfy the relations
(R1)-(R3) defining $W_{\Gamma'}$. We denote the values of $m_{ij}$ for
$\Gamma'$ by $m'_{ij}$.
Note that (R1) is clear, since $(s_ks_is_k)^2=s_ks_i^2s_k=e$, so it
remains to check (R2) and (R3). We will use the previous section to
reduce the number of checks that have to be done.

We first deal with some easy cases:

\begin{lemma} \label{l:easycases}
Let $i,j$ be distinct vertices of $\Gamma$.
\begin{enumerate}
\item[(a)] If $i=k$ or $j=k$ then $(t_it_j)^{m'_{ij}}=e$.
\item[(b)] If at most one of $i,j$ is connected to $k$ in
$\Gamma$ (or, equivalently, in $\Gamma'$), then $(t_it_j)^{m'_{ij}}=e.$
\end{enumerate}
\end{lemma}

\begin{proof}
For (a), suppose first that $i=k$.
Note that $m'_{ij}=m_{ij}$. The only non-trivial case is
thus when there is an arrow from $j$ to $k$, so that $t_j=s_ks_js_k$.
But then $t_it_j=s_ks_ks_js_k=s_js_k$ and the result follows. The case
when $j=k$ is similar.
For (b), suppose first that $i$ is connected to $k$.
The only non-trivial case is when the edge between $i$ and $k$ is
oriented towards $k$, so that $t_i=s_ks_is_k$. Then
$t_it_j=s_ks_is_ks_j$ and the result follows from the fact that
$s_ks_j=s_js_k$ (since $j$ is not connected to $k$).
The case when $j$ is connected to $k$ is similar.
\end{proof}

\begin{prop} \label{p:tirelations}
The elements $t_i$, for $i$ a vertex of $\Gamma$, satisfy the relations
(R1)-(R3) defining $W_{\Gamma'}$.
\end{prop}

\begin{proof}
After Lemma~\ref{l:easycases}, the remaining relations to check are
the relations from (R2), that $(t_it_j)^{m'_{ij}}=e$ when both $i$ and $j$
are connected to $k$, and the relations (R3) corresponding to the
chordless cycles in $\Gamma'$.
For the former, we need to check that, in going (i) from left to right
or (ii) from right to left in Corollary~\ref{c:localmutation}
(in each case, from $\Gamma$ to $\Gamma'$) we have that $(t_it_j)^{m'_{ij}}=e$.
(Note that it is enough to check this for the given labelling of $i,j$,
since the relation for the pair $(t_j,t_i)$ follows from the relation
for $(t_i,t_j)$.)
At the same time, we will check that the cycle relations in (R3) hold
if any cycles are produced; we shall denote by $r'(i,j)$, etc., the
elements in the cycle relations for $W_{\Gamma'}$.
Note that by Lemmas~\ref{l:ncyclesymmetry},~\ref{l:3cyclesymmetry}
and~\ref{l:4cyclesymmetry}, we only have to check one of the relations
of the form $r(i_a,i_a+1)^2=e$ arising from a cycle in order to be sure
that all of the relations corresponding to that cycle hold.

\begin{list}{}{\leftmargin=40pt\labelsep=0pt\labelwidth=35pt}
\item[(a)(i)\hfill] We have $(t_it_j)^2=(s_ks_is_ks_ks_js_k)^2=s_k(s_is_j)^2s_k=e$.
\item[(a)(ii)\hfill] We have $(t_it_j)^2=(s_is_j)^2=e$.
\item[(b)(i)\hfill] We have: $t_it_j= s_ks_is_ks_j= s_is_ks_is_j= s_i s_k s_j s_i$ and therefore
$s_i(t_it_j)^3 s_i= (s_k s_j)^3=e$.
And $r'(j,k)^2=(t_jt_kt_it_k)^2=(s_js_ks_ks_is_ks_k)^2=(s_js_i)^2=e$.
\item[(b)(ii)\hfill] We have
$(t_it_j)^2=(s_is_ks_js_k)^2=r(i,k)^2=e$.
\item[(c)(i)\hfill] We have $(t_it_j)^2=(s_ks_is_ks_ks_js_k)^2=s_k(s_is_j)^2s_k=e$.
\item[(c)(ii)\hfill] We have $(t_it_j)^2=(s_is_j)^2=e$.
\item[(d)(i)\hfill]
We have $s_k (t_it_j)^4 s_k = s_k (s_k s_i s_k s_j)^4 s_k = (s_i s_k s_j s_k)^4=(s_i s_j s_k s_j)^4$
and therefore
$s_j s_k (t_i t_j)^4 s_k s_j = (s_j s_i s_j s_k)^4=(s_i s_k)^4=e$.
We also have $r'(j,k)=(t_j t_k t_i t_k)^2 =(s_j s_k s_k s_i s_k s_k)^2=(s_j s_i)^2=e$.
\item[(d)(ii)\hfill] We have:
$(t_it_j)^2=(s_is_ks_js_k)^2=r(i,k)^2=e$.
\item[(e)\hfill]
Since the $s_i$ satisfy relations (R1), (R2) and (R3) with respect to
$\Gamma$, they also satisfy these relations with respect to $\Gamma^{\op}$,
by Proposition~\ref{p:opposite}.
For each vertex $i$ of $\Gamma$, let $u_i=s_kt_is_k$.
Then it is easy to check that we have:
\begin{equation*}
u_i=\begin{cases} s_ks_is_k & \text{if }k\rightarrow i; \\ s_i & \text{else}. \end{cases}.
\end{equation*}
If the induced subgraph on vertices $i,j,k$ is as in case (e) for $\Gamma$,
then it is as in case (d) for $\Gamma^{\op}$. Applying case (d) to the $u_i$
we see that the $u_i$ satisfy (R1),(R2) and
(R3) for $\mu_k(\Gamma^{\op})=(\mu_k(\Gamma))^{\op}$, and hence so do the
$t_i$, by the definition of the $u_i$. Thus case (e) follows from case (d).
\item[(f)(i)\hfill]
We have:
$s_j(t_it_j)^3s_j = s_j(s_ks_is_ks_j)^3s_j
= (s_js_ks_is_k)^3
= r(j,k)^3=e$.
And
$s_k r'(i,j)^2 s_k = s_k (t_it_jt_kt_j)^2 s_k
= s_k(s_k s_i s_k s_j s_k s_j)^2 s_k
= (s_i s_k s_j s_k s_j s_k)^2
= (s_is_js_ks_j)^2=r(i,j)^2=e$.
\item[(f)(ii)\hfill]
The argument is the same as for (f)(i) by symmetry.
\end{list}

Finally, we need to check that the elements $t_i$ satisfy the cycle relations
(R3) for all the chordless cycles in $\Gamma'$.
Note that Lemma~\ref{l:newcycles}
describes how the chordless cycles
in $\Gamma'$ arise from induced subdiagrams of
$\Gamma$. In each case, we need to check that the corresponding cycle
relations hold (using Lemmas~\ref{l:ncyclesymmetry},~\ref{l:3cyclesymmetry}
and~\ref{l:4cyclesymmetry} to reduce the work).
Note that, in the above, we have already checked
the cases (a)--(d) in Lemma~\ref{l:newcycles}. We now check the remaining cases.
We also note that, arguing as in case (e) above
we see that we only have to check cases (e), (g), (i), (j), (k) and
(l) in Lemma~\ref{l:newcycles}.
We choose a labelling of the vertices in each case.

\begin{list}{}{\leftmargin=40pt\labelsep=0pt\labelwidth=35pt}
\item[(e)\hfill] \
We label the vertices as follows.
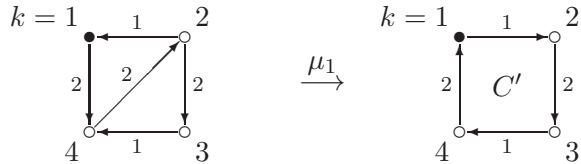
\begin{center}
\begin{picture}(200,60)
  \put(40,10){
    \put(140,18){
      \put(-18,-18){\circle{4}}
      \put(18,-18){\circle{4}}
      \put(-18,18){\circle*{4}}
      \put(18,18){\circle{4}}
      \put(-22,22){\RBCenter{$k=1$}}
      \put(22,22){$2$}
      \put(22,-22){\LTCenter{$3$}}
      \put(-22,-22){\RTCenter{$4$}}
      \put(-15,18){\vector(1,0){30}}
      \put(-18,-15){\vector(0,1){30}}
      \put(18,15){\vector(0,-1){30}}
      \put(15,-18){\vector(-1,0){30}}
      \put(-23,0){\HVCenter{\scriptsize $2$}}
      \put(23,0){\HVCenter{\scriptsize $2$}}
      \put(0,21){\HBCenter{\scriptsize $1$}}
      \put(0,-21){\HTCenter{\scriptsize $1$}}
      \put(0,0){\HVCenter{$C'$}}
    }
    \put(70,18){
      \put(0,0){\HVCenter{$\longrightarrow$}}
      \put(0,6){\HBCenter{$\mu_1$}}
    }
    \put(0,18){
      \put(-18,-18){\circle{4}}
      \put(18,-18){\circle{4}}
      \put(-18,18){\circle*{4}}
      \put(18,18){\circle{4}}
      \put(-22,22){\RBCenter{$k=1$}}
      \put(22,22){$2$}
      \put(22,-22){\LTCenter{$3$}}
      \put(-22,-22){\RTCenter{$4$}}
      \put(15,18){\vector(-1,0){30}}
      \put(-18,15){\vector(0,-1){30}}
      \put(18,15){\vector(0,-1){30}}
      \put(15,-18){\vector(-1,0){30}}
      \put(-18,-18){\put(2,2){\vector(1,1){32}}}
      \put(-23,0){\HVCenter{\scriptsize $2$}}
      \put(23,0){\HVCenter{\scriptsize $2$}}
      \put(0,21){\HBCenter{\scriptsize $1$}}
      \put(0,-21){\HTCenter{\scriptsize $1$}}
      \put(-4,4){\HVCenter{\scriptsize $2$}}
    }
  }
\end{picture}
\end{center}
\sloppy
We have
$r'(1,2)= t_1t_2t_3t_4t_3t_2=s_1 s_1 s_2 s_1 s_3 s_4 s_3 s_1 s_2 s_1$ and therefore
$s_1 s_2 s_1 r'(1,2) s_1 s_2 s_1= (s_1 s_2 s_1 s_2 s_1) s_3 s_4 s_3= s_2 s_3 s_4 s_3=r(2,3).$
Hence $r(2,3)^2=e$ implies $r'(1,2)^2=e$.
\item[(g)\hfill] \
We label the vertices as follows.
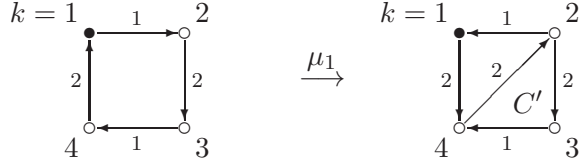
\begin{center}
\begin{picture}(200,60)
  \put(40,10){
    \put(0,18){
      \put(-18,-18){\circle{4}}
      \put(18,-18){\circle{4}}
      \put(-18,18){\circle*{4}}
      \put(18,18){\circle{4}}
      \put(-22,22){\RBCenter{$k=1$}}
      \put(22,22){$2$}
      \put(22,-22){\LTCenter{$3$}}
      \put(-22,-22){\RTCenter{$4$}}
      \put(-15,18){\vector(1,0){30}}
      \put(-18,-15){\vector(0,1){30}}
      \put(18,15){\vector(0,-1){30}}
      \put(15,-18){\vector(-1,0){30}}
      \put(-23,0){\HVCenter{\scriptsize $2$}}
      \put(23,0){\HVCenter{\scriptsize $2$}}
      \put(0,21){\HBCenter{\scriptsize $1$}}
      \put(0,-21){\HTCenter{\scriptsize $1$}}
    }
    \put(70,18){
      \put(0,0){\HVCenter{$\longrightarrow$}}
      \put(0,6){\HBCenter{$\mu_1$}}
    }
    \put(140,18){
      \put(-18,-18){\circle{4}}
      \put(18,-18){\circle{4}}
      \put(-18,18){\circle*{4}}
      \put(18,18){\circle{4}}
      \put(-22,22){\RBCenter{$k=1$}}
      \put(22,22){$2$}
      \put(22,-22){\LTCenter{$3$}}
      \put(-22,-22){\RTCenter{$4$}}
      \put(15,18){\vector(-1,0){30}}
      \put(-18,15){\vector(0,-1){30}}
      \put(18,15){\vector(0,-1){30}}
      \put(15,-18){\vector(-1,0){30}}
      \put(-18,-18){\put(2,2){\vector(1,1){32}}}
      \put(-23,0){\HVCenter{\scriptsize $2$}}
      \put(23,0){\HVCenter{\scriptsize $2$}}
      \put(0,21){\HBCenter{\scriptsize $1$}}
      \put(0,-21){\HTCenter{\scriptsize $1$}}
      \put(-4,4){\HVCenter{\scriptsize $2$}}
      \put(8,-8){\HVCenter{$C'$}}
    }
  }
\end{picture}
\end{center}
We have
$r'(4,3)=t_4 t_3 t_2 t_3= s_1 s_4 s_1 s_3 s_2 s_3$.
Using that $s_1$ and $s_3$ commute we get
$s_1 r'(4,3) s_1=s_4 s_1 s_3 s_2 s_3 s_1=s_4 s_3 (s_1 s_2 s_1) s_3=s_4 s_3 s_2 s_1 s_2 s_3=r(4,3)$ and hence $r(4,3)^ 2=e$ implies that $r'(4,3)^ 2=e$.
\item[(i)\hfill] \
We use the following labelling of the vertices.
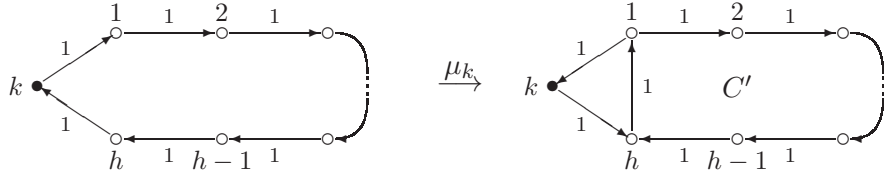
\begin{center}
\begin{picture}(346,65)
\put(56,12){
  \put(0,20){
    \multiput(0,-20)(0,40){2}{\circle{4}}
    \multiput(40,-20)(0,40){2}{\circle{4}}
    \multiput(80,-20)(0,40){2}{\circle{4}}
    \put(-30,0){\circle*{4}}
    \put(-38,0){\HVCenter{\small $k$}}
    \put(0,28){\HVCenter{\small $1$}}
    \put(40,28){\HVCenter{\small $2$}}
    \put(0,-28){\HVCenter{\small $h$}}
    \put(40,-28){\HVCenter{\small $h-1$}}
    \multiput(0,20)(40,0){2}{\put(3,0){\vector(1,0){34}}}
    \multiput(40,-20)(40,0){2}{\put(-3,0){\vector(-1,0){34}}}
    \put(0,-20){\put(-2.4,1.8){\vector(-3,2){26.2}}}
    \put(-30,0){\put(2.4,1.8){\vector(3,2){26.2}}}
    \put(80,20){\qbezier(3,0)(15,0)(15,-15)}
    \put(80,-20){\qbezier(3,0)(15,0)(15,15)\put(5,0){\vector(-1,0){2}}}
    \multiput(95,-3.5)(0,3){3}{\line(0,1){1}}
    \put(-19,14.3){\HVCenter{\scriptsize $1$}}
    \put(-19,-14.3){\HVCenter{\scriptsize $1$}}
    \multiput(20,26.3)(40,0){2}{\HVCenter{\scriptsize $1$}}
    \multiput(20,-26.3)(40,0){2}{\HVCenter{\scriptsize $1$}}
  }
  \put(130,20){
    \put(0,0){\HVCenter{$\longrightarrow$}}
    \put(0,7){\HVCenter{$\mu_k$}}
  }
  \put(195,20){
    \multiput(0,-20)(0,40){2}{\circle{4}}
    \multiput(40,-20)(0,40){2}{\circle{4}}
    \multiput(80,-20)(0,40){2}{\circle{4}}
    \put(-30,0){\circle*{4}}
    \put(-38,0){\HVCenter{\small $k$}}
    \put(0,28){\HVCenter{\small $1$}}
    \put(40,28){\HVCenter{\small $2$}}
    \put(0,-28){\HVCenter{\small $h$}}
    \put(40,-28){\HVCenter{\small $h-1$}}
    \multiput(0,20)(40,0){2}{\put(3,0){\vector(1,0){34}}}
    \multiput(40,-20)(40,0){2}{\put(-3,0){\vector(-1,0){34}}}
    \put(0,-17){\vector(0,1){34}}
    \put(0,20){\put(-2.4,-1.8){\vector(-3,-2){26.2}}}
    \put(-30,0){\put(2.4,-1.8){\vector(3,-2){26.2}}}
    \put(80,20){\qbezier(3,0)(15,0)(15,-15)}
    \put(80,-20){\qbezier(3,0)(15,0)(15,15)\put(5,0){\vector(-1,0){2}}}
    \multiput(95,-3.5)(0,3){3}{\line(0,1){1}}
    \put(-19,14.3){\HVCenter{\scriptsize $1$}}
    \put(-19,-14.3){\HVCenter{\scriptsize $1$}}
    \multiput(20,26.3)(40,0){2}{\HVCenter{\scriptsize $1$}}
    \multiput(20,-26.3)(40,0){2}{\HVCenter{\scriptsize $1$}}
    \put(6,0){\HVCenter{\scriptsize $1$}}
    \put(40,0){\HVCenter{$C'$}}
  }
}
\end{picture}
\end{center}

Note that $t_h = s_k s_h s_k$, while $t_i=s_i$ for all $i\not=h$.
\begin{align*}
r'(1,2)^2 &= (t_1t_2\cdots t_{h-1}t_h t_{h-1}\cdots t_2)^2 \\
&= (s_1s_2\cdots s_{h-1} s_k s_h s_k s_{h-1}\cdots s_2)^2 \\
&= (s_1s_2\cdots s_{h-1}s_h s_k s_h s_{h-1}\cdots s_2)^2 = r(1,2)^2 = e.
\end{align*}

\item[(j)\hfill] \
We label the vertices as follows.
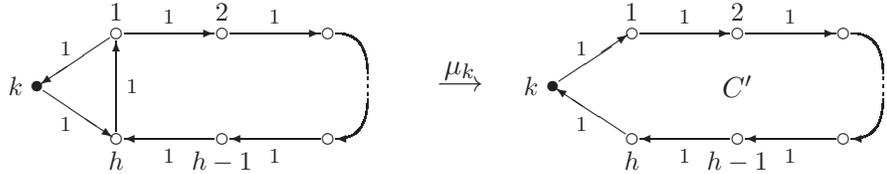
\begin{center}
\begin{picture}(346,65)
\put(56,12){
  \put(195,20){
    \multiput(0,-20)(0,40){2}{\circle{4}}
    \multiput(40,-20)(0,40){2}{\circle{4}}
    \multiput(80,-20)(0,40){2}{\circle{4}}
    \put(-30,0){\circle*{4}}
    \put(-38,0){\HVCenter{\small $k$}}
    \put(0,28){\HVCenter{\small $1$}}
    \put(40,28){\HVCenter{\small $2$}}
    \put(0,-28){\HVCenter{\small $h$}}
    \put(40,-28){\HVCenter{\small $h-1$}}
    \multiput(0,20)(40,0){2}{\put(3,0){\vector(1,0){34}}}
    \multiput(40,-20)(40,0){2}{\put(-3,0){\vector(-1,0){34}}}
    \put(0,-20){\put(-2.4,1.8){\vector(-3,2){26.2}}}
    \put(-30,0){\put(2.4,1.8){\vector(3,2){26.2}}}
    \put(80,20){\qbezier(3,0)(15,0)(15,-15)}
    \put(80,-20){\qbezier(3,0)(15,0)(15,15)\put(5,0){\vector(-1,0){2}}}
    \multiput(95,-3.5)(0,3){3}{\line(0,1){1}}
    \put(-19,14.3){\HVCenter{\scriptsize $1$}}
    \put(-19,-14.3){\HVCenter{\scriptsize $1$}}
    \multiput(20,26.3)(40,0){2}{\HVCenter{\scriptsize $1$}}
    \multiput(20,-26.3)(40,0){2}{\HVCenter{\scriptsize $1$}}
    \put(40,0){\HVCenter{$C'$}}
  }
  \put(130,20){
    \put(0,0){\HVCenter{$\longrightarrow$}}
    \put(0,7){\HVCenter{$\mu_k$}}
  }
  \put(0,20){
    \multiput(0,-20)(0,40){2}{\circle{4}}
    \multiput(40,-20)(0,40){2}{\circle{4}}
    \multiput(80,-20)(0,40){2}{\circle{4}}
    \put(-30,0){\circle*{4}}
    \put(-38,0){\HVCenter{\small $k$}}
    \put(0,28){\HVCenter{\small $1$}}
    \put(40,28){\HVCenter{\small $2$}}
    \put(0,-28){\HVCenter{\small $h$}}
    \put(40,-28){\HVCenter{\small $h-1$}}
    \multiput(0,20)(40,0){2}{\put(3,0){\vector(1,0){34}}}
    \multiput(40,-20)(40,0){2}{\put(-3,0){\vector(-1,0){34}}}
    \put(0,-17){\vector(0,1){34}}
    \put(0,20){\put(-2.4,-1.8){\vector(-3,-2){26.2}}}
    \put(-30,0){\put(2.4,-1.8){\vector(3,-2){26.2}}}
    \put(80,20){\qbezier(3,0)(15,0)(15,-15)}
    \put(80,-20){\qbezier(3,0)(15,0)(15,15)\put(5,0){\vector(-1,0){2}}}
    \multiput(95,-3.5)(0,3){3}{\line(0,1){1}}
    \put(-19,14.3){\HVCenter{\scriptsize $1$}}
    \put(-19,-14.3){\HVCenter{\scriptsize $1$}}
    \multiput(20,26.3)(40,0){2}{\HVCenter{\scriptsize $1$}}
    \multiput(20,-26.3)(40,0){2}{\HVCenter{\scriptsize $1$}}
    \put(6,0){\HVCenter{\scriptsize $1$}}
  }
}
\end{picture}
\end{center}
Here we have $t_1=s_k s_1 s_k$ and $t_i=s_i$ for all $i\neq 1$.
\begin{align*}
r'(h,h-1)^2 &= (t_h t_{h-1} \cdots t_1 t_k t_1 \cdots t_{h-1})^2 \\
&= (s_h s_{h-1} \cdots s_{2} s_k s_1 (s_k s_k) s_k s_1 s_k s_{2}\cdots s_{h-1})^2
\\
&= (s_h s_{h-1} \cdots s_{2} (s_k s_1 s_k s_1 s_k) s_{2}\cdots s_{h-1})^2 \\
&= (s_h s_{h-1}\cdots s_{2} s_1 s_{2} \cdots s_{h-1})^2 = r(h,h-1)^ 2 = e.
\end{align*}

\item[(k)\hfill]
This is trivial.

\item[(l)\hfill]
Suppose that $k$ is connected to exactly one vertex of $C$ via an
edge of unspecified weight.
We label the vertices as follows.
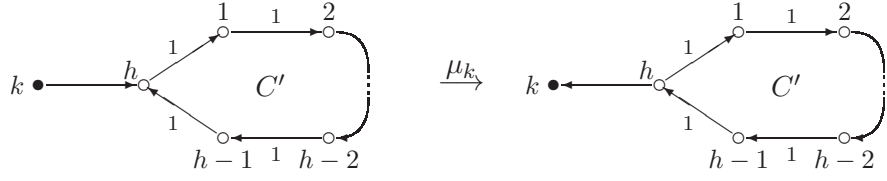
\begin{center}
\begin{picture}(346,70)
\put(56,12){
  \put(0,20){
    \put(40,0){
    \multiput(0,-20)(0,40){2}{\circle{4}}
    \multiput(40,-20)(0,40){2}{\circle{4}}
    \put(-30,0){\circle{4}}
    \put(-70,0){\circle*{4}}
    \put(-35,6){\HVCenter{\small $h$}}
    \put(-78,0){\HVCenter{\small $k$}}
    \put(0,28){\HVCenter{\small $1$}}
    \put(40,28){\HVCenter{\small $2$}}
    \put(0,-28){\HVCenter{\small $h-1$}}
    \put(40,-28){\HVCenter{\small $h-2$}}
    \put(0,20){\put(3,0){\vector(1,0){34}}}
    \put(-70,0){\put(3,0){\vector(1,0){34}}}
    \put(40,-20){\put(-3,0){\vector(-1,0){34}}}
    \put(0,-20){\put(-2.4,1.8){\vector(-3,2){26.2}}}
    \put(-30,0){\put(2.4,1.8){\vector(3,2){26.2}}}
    \put(40,20){\qbezier(3,0)(15,0)(15,-15)}
    \put(40,-20){\qbezier(3,0)(15,0)(15,15)\put(5,0){\vector(-1,0){2}}}
    \multiput(55,-3.5)(0,3){3}{\line(0,1){1}}
    \put(-19,14.3){\HVCenter{\scriptsize $1$}}
    \put(-19,-14.3){\HVCenter{\scriptsize $1$}}
    \multiput(20,26.3)(0,-52.6){2}{\HVCenter{\scriptsize $1$}}
    \put(18,0){\HVCenter{$C'$}}
    }
  }
  \put(130,20){
    \put(0,0){\HVCenter{$\longrightarrow$}}
    \put(0,7){\HVCenter{$\mu_k$}}
  }
  \put(195,20){
    \put(40,0){
    \multiput(0,-20)(0,40){2}{\circle{4}}
    \multiput(40,-20)(0,40){2}{\circle{4}}
    \put(-30,0){\circle{4}}
    \put(-70,0){\circle*{4}}
    \put(-35,6){\HVCenter{\small $h$}}
    \put(-78,0){\HVCenter{\small $k$}}
    \put(0,28){\HVCenter{\small $1$}}
    \put(40,28){\HVCenter{\small $2$}}
    \put(0,-28){\HVCenter{\small $h-1$}}
    \put(40,-28){\HVCenter{\small $h-2$}}
    \put(0,20){\put(3,0){\vector(1,0){34}}}
    \put(-30,0){\put(-3,0){\vector(-1,0){34}}}
    \put(40,-20){\put(-3,0){\vector(-1,0){34}}}
    \put(0,-20){\put(-2.4,1.8){\vector(-3,2){26.2}}}
    \put(-30,0){\put(2.4,1.8){\vector(3,2){26.2}}}
    \put(40,20){\qbezier(3,0)(15,0)(15,-15)}
    \put(40,-20){\qbezier(3,0)(15,0)(15,15)\put(5,0){\vector(-1,0){2}}}
    \multiput(55,-3.5)(0,3){3}{\line(0,1){1}}
    \put(-19,14.3){\HVCenter{\scriptsize $1$}}
    \put(-19,-14.3){\HVCenter{\scriptsize $1$}}
     \multiput(20,26.3)(0,-52.6){2}{\HVCenter{\scriptsize $1$}}
    \put(18,0){\HVCenter{$C'$}}
    }
  }
}
\end{picture}
\end{center}
If the edge between $k$ and $h$ points towards $h$, then $t_h=s_h$ and
the relations for $C'$ are the same as the relations for $C$.
If the edge points towards $k$, then $t_h=s_ks_hs_k$ and we have:
\begin{align*}
r'(1,2)^2 &= (t_1t_2\cdots t_h t_{h-1}\cdots t_2)^2 \\
&= (s_1s_2\cdots s_{h-1}s_k s_h s_k s_{h-1}\cdots s_2)^2 \\
&= s_k(s_1s_2\cdots s_{h-1} s_h s_{h-1}\cdots s_2)^2 s_k \\
&= s_kr(1,2)^2s_k = e.
\end{align*}
\end{list}
This finishes the proof.
\end{proof}

Note that, in the simply-laced case, some arguments of this kind
have been used in the proof of~\cite[Thm.\ 6.10]{cst}. We explain more about the
connection with~\cite{cst} at the end of Section~\ref{s:companionbasis}.

Recall that we are in the situation where $\Gamma'=\mu_k(\Gamma)$.
We denote by $s'_i$ the defining generators of $W_{\Gamma'}$.

For each vertex $i$ of $\Gamma$ define the element $t'_i$ in
$W_{\Gamma'}$ as follows:
$$t'_i=\begin{cases}
s'_k s'_i s'_k &
\text{if there is an arrow $k\rightarrow i$ in $\Gamma'$
(of arbitrary weight);} \\
s'_i & \text{otherwise.}
\end{cases}
$$

\begin{prop} \label{p:titilderelations}
The elements $t'_i$, for $i$ a vertex of $\Gamma'$, satisfy the relations
(R1)-(R3) defining $W_{\Gamma}$.
\end{prop}

\begin{proof}
This follows from Proposition~\ref{p:tirelations}, interchanging $\Gamma$
and $\Gamma'$ and using the fact that the definition of the group $W_{\Gamma}$
is unchanged under reversing the orientation of all the arrows in $\Gamma$
(see Proposition~\ref{p:opposite}).
\end{proof}

We can now prove the following:

\begin{theorem} \label{t:isomorphism}
\begin{enumerate}
\item[(a)] Let $\Gamma$ be a diagram of finite type and $\Gamma'=\mu_k(\Gamma)$ the
mutation of $\Gamma$ at a vertex $k$.
Then $W_{\Gamma}\cong W_{\Gamma'}$.
\item[(b)]
Let $\mathcal{A}$ be a cluster algebra of finite type. Then the groups
$W_{\Gamma}$ associated to the diagrams $\Gamma$ arising from the seeds of
$\mathcal{A}$ are all isomorphic (to the reflection group associated to the
Dynkin diagram associated to $\mathcal{A}$).
\end{enumerate}
\end{theorem}

\begin{proof}
Denote the generators of $W_{\Gamma'}$
(as defined in Section~\ref{s:groupdefinition}) by $s'_i$, for $1\leq i\leq n$.
By Proposition~\ref{p:tirelations}, there is a group
homomorphism $\varphi:W_{\Gamma'}\rightarrow W_{\Gamma}$ such that
$\varphi(s'_i)=t_i$ for all $i$.
By Proposition~\ref{p:titilderelations}
there is a group homomorphism $\psi:W_{\Gamma}\rightarrow W_{\Gamma'}$ such that
$\psi(s_i)=t'_i$.
If there is no arrow $i\rightarrow k$ in $\Gamma$, then there is also no arrow
$k\rightarrow i$ in $\Gamma'$ and consequently
$$\psi(\varphi(s'_i))=\psi(t_i)=\psi(s_i)=t'_i=s'_i.$$
If there is an arrow $i\rightarrow k$ in $\Gamma$, then there is an arrow $k\rightarrow i$ in $\Gamma'$ and therefore
\begin{align*}
\psi(\varphi(s'_i)) &= \psi(t_i) \\
&= \psi(s_ks_is_k) \\
&= \psi(s_k)\psi(s_i)\psi(s_k) \\
&= t'_kt'_it'_k \\
&= s'_ks'_ks'_is'_ks'_k = s'_i.
\end{align*}
So $\psi \circ \varphi=\id_{W_{\Gamma'}}$, and, similarly,
$\varphi \circ \psi=\id_{W_{\Gamma}}$, showing that $\varphi$ and $\psi$ are
isomorphisms and (a) is proved. Part (b) follows from part (a) and
Remark~\ref{r:Dynkincase}.
\end{proof}

\section{Companion Bases} \label{s:companionbasis}
We first recall some results from~\cite{bgz}.
Recall that a matrix $A$ is said to be \emph{symmetrisable} if there
is diagonal matrix $D$, with positive entries on the diagonal, such that $DA$ is symmetric.
A symmetrisable matrix $A$ is said to be \emph{quasi-Cartan} if its
diagonal entries are all equal to $2$, and it is said to be \emph{positive} if
$DA$ is positive definite for some positive diagonal matrix $D$ (this does not
depend on the choice of $D$).
If $B$ is a skew-symmetrisable matrix, a \emph{quasi-Cartan companion}
of $B$ is a quasi-Cartan matrix $A$ such that $|A_{ij}|=|B_{ij}|$ for all
$i\neq j$.
Such matrices are used in~\cite{bgz} to give a characterisation of cluster
algebras of finite type:

\begin{theorem} \cite{bgz}
Let $B$ be a skew-symmetrisable matrix. Then the cluster algebra
$\mathcal{A}(B)$ is of finite type if and only if every chordless cycle in $\Gamma(B)$ is
cyclically oriented, and $B$ has a positive quasi-Cartan companion.
\end{theorem}

We also mention the following result, which we shall need later:

\begin{prop} \cite[1.4]{bgz} \label{p:cyclecondition}
Let $A$ be a positive quasi-Cartan companion of a \linebreak
skew-symmetrisable
matrix $B$. Then, for every chordless cycle
$$i_0\rightarrow i_1\rightarrow \cdots \rightarrow i_{d-1}\rightarrow i_0$$
in $\Gamma(B)$, the following holds:
$$\prod_{a=0}^{d-1} (-A_{i_a,i_{a+1}}) < 0.$$
\end{prop}

Let $\Phi$ be a root system of rank $n$ associated to a Dynkin diagram $\Delta$
and let $B$ be the exchange matrix in a seed $(\mathbf{x},B)$ of the cluster
algebra of type $\Delta$. Motivated by~\cite{bgz},
Parsons~\cite[Defn.\ 4.1]{parsons}
defines a \emph{companion basis} of $B$ to be a subset
$\beta_1,\beta_2,\ldots ,\beta_n$ of $\Phi$ which is a $\mathbb{Z}$-basis of
$\mathbb{Z}\Phi$ such that the matrix $A$ with $i,j$ entry
$A_{ij}=(\beta_i,\beta_j)$ for all $i,j$ is a quasi-Cartan companion of $B$.
In the general case (not necessarily simply-laced),
we use $A_{ij}=(\beta_i,\beta^{\vee}_j)$.
Parsons uses companion bases to obtain the dimension vectors of the indecomposable
modules over the cluster-tilted algebra associated to $(\mathbf{x},B)$ by~\cite{bmrrt}
and~\cite{ccs}, in the type $A$ case (they are obtained by expanding arbitrary roots in terms of
the basis; see~\cite[Theorem 5.3]{parsons}).
We remark that independent later
work of Ringel~\cite{ringel} also obtains such dimension vectors (in a more general setting which
includes all the finite type cases).

Suppose that $\B=\{\beta_1,\beta_2,\ldots ,\beta_n\}$ is a companion basis for
an exchange matrix $B$ in a cluster algebra of finite type as above, and fix $1\leq k\leq n$.
Parsons~\cite[Thm.\ 6.1]{parsons} defines the (inward) \emph{mutation} $\B'=\mu_k(\B)$ of
$\B$ at $k$ to be the subset $\{\beta'_1,\beta'_2,\ldots ,\beta'_n\}$ of $\Phi$ defined by

$$\beta'_i=\begin{cases}
s_{\beta_k}(\beta_i) & \text{if there is an arrow from $i$ to $k$ in $\Gamma(B)$;} \\
\beta_i & \text{otherwise.}
\end{cases}$$

Outward mutation of $\B$ is defined similarly.

\begin{prop} \cite[Thm.\ 6.1]{parsons}
\label{p:mutatecompanionsimplylaced}
Let $\B$ be a companion basis for an exchange matrix $B$ in a cluster
algebra of finite simply-laced type and fix $1\leq k\leq n$.
Then $\mu_k(\B)$ is a companion basis for $\mu_k(B)$.
\end{prop}

We remark that mutations of quasi-Cartan companion matrices were considered
in~\cite{bgz}.

Proposition~\ref{p:mutatecompanionsimplylaced} can be extended to all finite
type cases using a similar argument:

\begin{prop} \label{p:mutatecompanion}
Let $\B$ be a companion basis for an exchange matrix $B$ in a cluster
algebra of finite type and fix $1\leq k\leq n$.
Then $\mu_k(\B)$ is a companion basis for $\mu_k(B)$.
\end{prop}

\begin{proof}
We must check that if $\B=\{\beta_1,\beta_2,\ldots ,\beta_n\}$ is
a companion basis for an exchange matrix $B$ underlying a diagram $\Gamma$
then $\B'=\{\gamma_1,\gamma_2,\ldots , \gamma_n\}=\mu_k(\B)$ is a companion
basis for the mutation $B'=\mu_k(B)$, which underlies $\mu_k(\Gamma)$.
Note that~\cite[Thm.\ 6.1]{parsons} covers the cases
$(\gamma_p,\gamma_q)$ where $p=k$ or $q=k$ or where at most one of $p$ and $q$
is connected to $k$ by an edge (in this case the assumption
of simply-laced type is not used).

Hence, since~\cite[Thm.\ 6.1]{parsons} covers the simply-laced
case, we are reduced to cases (d),(e) and (f) in
Corollary~\ref{c:localmutation}. In each case, we are only left with checking
that $|(\gamma_i,\gamma_j)|=|B'_{ij}|$.

We denote by $A$ the $3\times 3$ matrix
$$
\begin{pmatrix}
(\beta_i,\beta_i^{\vee}) & (\beta_i,\beta_k^{\vee}) & (\beta_i,\beta_j^{\vee}) \\
(\beta_k,\beta_i^{\vee}) & (\beta_k,\beta_k^{\vee}) & (\beta_k,\beta_j^{\vee}) \\
(\beta_j,\beta_i^{\vee}) & (\beta_j,\beta_k^{\vee}) & (\beta_j,\beta_j^{\vee})
\end{pmatrix}.
$$

We consider first case (d), mutating from left to right.
There are two possibilities for $A$:
$$\begin{pmatrix}
2 & \pm 1 & 0 \\
\pm 2 & 2 & \pm 1 \\
0 & \pm 1 & 2
\end{pmatrix},
\begin{pmatrix}
2 & \pm 2 & 0 \\
\pm 1 & 2 & \pm 1 \\
0 & \pm 1 & 2
\end{pmatrix}.$$

In the first case, the mutation of $B$ is:
$$
\begin{pmatrix}
0 & 1 & 0 \\
-2 & 0 & 1 \\
0 & -1 & 0
\end{pmatrix}
\rightarrow
\begin{pmatrix}
0 & -1 & 1 \\
2 & 0 & -1 \\
-2 & 1 & 0
\end{pmatrix}.
$$
We have
\begin{align*}
(\gamma_i,\gamma_j^{\vee}) &= (s_{\beta_k}(\beta_i),\beta_j^{\vee}) \\
&= (\beta_i-(\beta_i,\beta_k^{\vee})\beta_k,\beta_j^{\vee}) \\
&= (\beta_i,\beta_j^{\vee})-(\beta_i,\beta_k^{\vee})(\beta_k,\beta_j^{\vee}) = \pm (\beta_i,\beta_k^{\vee}) = \pm 1,
\end{align*}
and
\begin{align*}
(\gamma_j,\gamma_i^{\vee}) &= (\beta_j,(s_{\beta_k}(\beta_i))^{\vee}) \\
&= (s_{\beta_k}(\beta_j),\beta_i^{\vee}) \\
&= (\beta_j-(\beta_j,\beta_k^{\vee})\beta_k,\beta_i^{\vee}) \\
&= (\beta_j,\beta_i^{\vee})-(\beta_j,\beta_k^{\vee})(\beta_k,\beta_i^{\vee}) = \pm (\beta_k,\beta_i^{\vee}) = \pm 2,
\end{align*}
and this case is done.

In the second case, the mutation is:
$$
\begin{pmatrix}
0 & 2 & 0 \\
-1 & 0 & 1 \\
0 & -1 & 0
\end{pmatrix}
\to
\begin{pmatrix}
0 & -2 & 2 \\
1 & 0 & -1 \\
-1 & 1 & 0
\end{pmatrix}
$$
Arguing as above, we have $(\gamma_i,\gamma_j^{\vee})=\pm 2$ and
$(\gamma_j,\gamma_i^{\vee})=\pm 1$ as required.

Next, we consider case (d), mutating from right to left. Note that the
skew-symmetrisability of $B$ imposes restrictions; see also
\cite[Lemma 7.6]{fominzelevinsky2}. The only possibilities for $B$ are thus
$$
\begin{pmatrix}
0 & -1 & 1 \\
2 & 0 & -1 \\
-2 & 1 & 0
\end{pmatrix}
\quad\text{and}\quad
\begin{pmatrix}
0 & -2 & 2 \\
1 & 0 & -1 \\
-1 & 1 & 0
\end{pmatrix}.
$$
In the first case, the mutation is
$$
\begin{pmatrix}
0 & -1 & 1 \\
2 & 0 & -1 \\
-2 & 1 & 0
\end{pmatrix}
\to
\begin{pmatrix}
0 & 1 & 0 \\
-2 & 0 & 1 \\
0 & -1 & 0
\end{pmatrix}.
$$
We have $(\gamma_i,\gamma_j^{\vee})=(\beta_i,\beta_j^{\vee})-(\beta_i,\beta_k^{\vee})(\beta_k,\beta_j^{\vee})$.
Moreover,
$|(\beta_i,\beta_j^{\vee})|=|(\beta_i,\beta_k^{\vee})|=|(\beta_k,\beta_j^{\vee})|=1$.
Note that the matrix $A$
must be positive (as the $\beta_i$ form a basis of $\mathbb{Z}\Phi$). Hence, by Proposition~\ref{p:cyclecondition}, an odd number of
the signs of $(\beta_i,\beta_j^{\vee})$, $(\beta_i,\beta_k^{\vee})$ and $(\beta_k,\beta_j^{\vee})$ must
be positive. We obtain the following table of possible values:

\begin{center}
\begin{tabular}{c|c|c|c}
$(\beta_i,\beta_j^{\vee})$ & $(\beta_i,\beta_k^{\vee})$ & $(\beta_k,\beta_j^{\vee})$ & $(\gamma_i,\gamma_j^{\vee})$ \\
\hline
$1$ & $1$ & $1$ & $0$ \\
$1$ & $-1$ & $-1$ & $0$ \\
$-1$ & $1$ & $-1$ & $0$ \\
$-1$ & $-1$ & $1$ & $0$ \\
\end{tabular}
\end{center}

In the second case, the mutation is
$$
\begin{pmatrix}
0 & -2 & 2 \\
1 & 0 & -1 \\
-1 & 1 & 0
\end{pmatrix}
\to
\begin{pmatrix}
0 & 2 & 0 \\
-1 & 0 & 1 \\
0 & -1 & 0
\end{pmatrix}.
$$
As before, $(\gamma_i,\gamma_j^{\vee})=(\beta_i,\beta_j^{\vee})-(\beta_i,\beta_k^{\vee})(\beta_k,\beta_j^{\vee})$.
We have $|(\beta_i,\beta_j^{\vee})|=|(\beta_i,\beta_k^{\vee})|=2$ and $|(\beta_k,\beta_j^{\vee})|=1$.
As in the previous case, we obtain a table of possible values:
\begin{center}
\begin{tabular}{c|c|c|c}
$(\beta_i,\beta_j^{\vee})$ & $(\beta_i,\beta_k^{\vee})$ & $(\beta_k,\beta_j^{\vee})$ & $(\gamma_i,\gamma_j^{\vee})$ \\
\hline
$2$ & $2$ & $1$ & $0$ \\
$2$ & $-2$ & $-1$ & $0$ \\
$-2$ & $2$ & $-1$ & $0$ \\
$-2$ & $-2$ & $1$ & $0$ \\
\end{tabular}
\end{center}
and case (d) is complete.

For case (e), we note that replacing $B$ with $-B$ (and keeping the
$\beta_i$ unchanged) gives an instance of case (d).
Hence, we have
$|\mu_k(-B)_{ij}|=|-B'_{ij}|=|(\gamma'_i,\gamma'_j)|$, where
$$\gamma'_i=\begin{cases} s_{\beta_k}(\beta_i) & i\rightarrow k \text{ in } \Gamma(-B); \\
\beta_i & \text{else}, \end{cases}$$
which coincides with $s_{\beta_k}(\gamma_i)$. Therefore
$$|B_{ij}|=|(s_{\beta_k}(\gamma_i),s_{\beta_k}(\gamma_i))|=|(\gamma_i,\gamma_j)|,$$
as required.

We are left with case
(f) to consider. By symmetry, we only need to consider the
mutation from the left hand side
of the figure to the right hand side. Using the skew-symmetrisability of $B$
(see \cite[Lemma 7.6]{fominzelevinsky2}) we conclude that there are only
two possible cases for $B$, namely
$$
\begin{pmatrix}
0 & 1 & -1 \\
-2 & 0 & 2 \\
1 & -1 & 0
\end{pmatrix}
\quad\text{and}\quad
\begin{pmatrix}
0 & 2 & -1 \\
-1 & 0 & 1 \\
1 & -2 & 0
\end{pmatrix}.
$$
In the first case, the mutation is
$$
\begin{pmatrix}
0 & 1 & -1 \\
-2 & 0 & 2 \\
1 & -1 & 0
\end{pmatrix}
\to
\begin{pmatrix}
0 & -1 & 1 \\
2 & 0 & -2 \\
-1 & 1 & 0
\end{pmatrix}.
$$

We have $(\gamma_i,\gamma_j^{\vee})=(s_{\beta_k}(\beta_i),\beta_j^{\vee})=
(\beta_i,\beta_j^{\vee})-(\beta_i,\beta_k^{\vee})(\beta_k,\beta_j^{\vee})$.
Again, using Proposition~\ref{p:cyclecondition}, we obtain the following table of possible values:

\begin{center}
\begin{tabular}{c|c|c|c}
$(\beta_i,\beta_j^{\vee})$ & $(\beta_i,\beta_k^{\vee})$ & $(\beta_k,\beta_j^{\vee})$ & $(\gamma_i,\gamma_j^{\vee})$ \\
\hline
$1$ & $1$ & $2$ & $-1$ \\
$1$ & $-1$ & $-2$ & $-1$ \\
$-1$ & $1$ & $-2$ & $1$ \\
$-1$ & $-1$ & $2$ & $1$ \\
\end{tabular}
\end{center}

A similar argument shows that $(\gamma_j,\gamma_i^{\vee})=\pm 1$.
In the second case, the mutation is:
$$
\begin{pmatrix}
0 & 2 & -1 \\
-1 & 0 & 1 \\
1 & -2 & 0
\end{pmatrix}
\to
\begin{pmatrix}
0 & -2 & 1 \\
1 & 0 & -1 \\
-1 & 2 & 0
\end{pmatrix}.
$$

As before, $(\gamma_i,\gamma_j^{\vee})=(\beta_i,\beta_j^{\vee})-(\beta_i,\beta_k^{\vee})(\beta_k,\beta_j^{\vee})$.
and we obtain the following table of possibilities:

\begin{center}
\begin{tabular}{c|c|c|c}
$(\beta_i,\beta_j^{\vee})$ & $(\beta_i,\beta_k^{\vee})$ & $(\beta_k,\beta_j^{\vee})$ & $(\gamma_i,\gamma_j^{\vee})$ \\
\hline
$1$ & $2$ & $1$ & $-1$ \\
$1$ & $-2$ & $-1$ & $-1$ \\
$-1$ & $2$ & $-1$ & $1$ \\
$-1$ & $-2$ & $1$ & $1$ \\
\end{tabular}
\end{center}

A similar argument shows that $(\gamma_j,\gamma_i^{\vee})=\pm 1$.
The proof is complete.
\end{proof}

\begin{remark} \label{r:mutatedall} \rm
Every companion basis for $\mu_k(B)$ is of the form $\mu_k(\B)$ for
some companion basis $\B$ for $B$, since it can be mutated back
to a companion basis for $B$ using outward mutation, and
outward and inward mutation are easily seen to be inverses of each other.
\end{remark}

\begin{corollary} \cite{bgz}
Let $B$ be an exchange matrix in a cluster algebra of finite type.
Then $B$ has a companion basis.
\end{corollary}
\begin{proof}
This follows from
results in~\cite{bgz} (see also~\cite[Cor.\ 3.10]{parsons}).
Alternatively, it can be proved using Proposition~\ref{p:mutatecompanion}.
\end{proof}

\begin{remark} \rm
It is not the case that every subset of $\Phi$ forming a basis for
$\mathbb{Z}\Phi$ is a companion basis, even if we just restrict to the
positive roots. For example, in type $A_4$, the set of roots
$\{\alpha_1,\alpha_2,\alpha_2+\alpha_3,\alpha_2+\alpha_3+\alpha_4\}$
is a such a subset, but is not a companion basis (the inner product
of any pair of distinct roots in this subset is $\pm 1$).
\end{remark}

We can now prove:

\begin{theorem}
Let $\B=\{\beta_1,\beta_2,\ldots ,\beta_n\}\subseteq \Phi$
be a companion basis for an exchange matrix $B$ in a cluster algebra
of finite type $\Delta$ (with associated root system $\Phi$).
Then there is an isomorphism between $W$ and $W_{\Gamma(B)}$
as defined in Section~\ref{s:groupdefinition} taking $s_{\beta_i}$ to
$s_i$. In particular, the defining relations for $W_{\Gamma(B)}$ can be
regarded as giving a presentation of $W$ in terms of the
generators $s_{\beta_1},s_{\beta_2},\ldots ,s_{\beta_n}$.
\end{theorem}

\begin{proof}
Note that any simple system in $\Phi$ is a companion basis for
an exchange matrix $B$ whose Cartan counterpart is a Cartan matrix
of type $\Delta$. By Remark~\ref{r:Dynkincase} it follows that
the result holds for this case.

Suppose that the result holds for
every companion basis associated to an exchange matrix $B$.
Fix $1\leq k\leq n$ and let $\B'=\{\beta'_1,\beta'_2,\ldots ,\beta'_n\}$
be a companion basis for $\mu_k(B)$. By Remark~\ref{r:mutatedall},
$\B'$ is of the form $\mu_k(\B)$ for some companion basis
$\B=\{\beta_1,\beta_2,\ldots ,\beta_n\}$ for $B$.

By assumption, the result holds for $\B$, i.e.\ there
is a group isomorphism $\tau_{\B}:W\rightarrow W_{\Gamma(B)}$ taking
$s_{\beta_i}$ to $s_i$ for $1\leq i\leq n$.
If there is no arrow from $i$ to $k$ in $\Gamma(B)$
then $\beta'_i=\beta_i$. So,
$\tau_{\B}(s_{\beta'_i})=\tau_{\B}(s_{\beta_i})=s_i=t_i$.
If there is an arrow from $i$ to $k$ in $\Gamma(B)$,
we have:
$$
s_{\beta'_i} = s_{s_{\beta_k}(\beta_i)}
= s_{\beta_k}s_{\beta_i}s^{-1}_{\beta_k}
= s_{\beta_k}s_{\beta_i}s_{\beta_k},
$$
so $\tau_{\B}(s_{\beta'_i})=s_ks_is_k=t_i$.

From the proof of Theorem~\ref{t:isomorphism}, there
is an isomorphism $\psi:W_{\Gamma(B)}\rightarrow W_{\Gamma(B')}$.
If the defining generators of $W_{\Gamma'}$ are denoted $s'_i$, then
we have $\psi(t_i)=s'_i$. Hence $\tau_{\B'}:=\psi\circ \tau_{\B}:W\rightarrow
W_{\Gamma'}$ is a group isomorphism taking $s_{\beta'_i}$ to $s'_i$.
Therefore, the result holds for $\B'$.

The result follows in general by using induction on the number of
mutations needed to obtain $B$ from a fixed exchange matrix whose
Cartan counterpart is a Cartan matrix of type $\Delta$.
\end{proof}

Finally, we explain how the perspective of companion bases allows an
alternative proof of Theorem~\ref{t:isomorphism} in the simply-laced case
using results from~\cite{cst}.
Let $B$ be an exchange matrix of a cluster algebra of finite simply-laced type
with corresponding diagram $\Gamma$.
Let $\B=\{\beta_1,\beta_2,\ldots ,\beta_n\}$ be a companion basis for $B$
and let $A$ be the corresponding matrix with $A_{ij}=(\beta_i,\beta_j)$.

Let $G$ be the unoriented graph on vertices $1,2,\ldots ,n$
with an edge between $i$ and $j$ whenever $A_{ij}\not=0$. We define a map
$f$ from the set of edges of $G$ to $\{1,-1\}$ by setting $f(\{i,j\})$ to
be the sign of $A_{ij}$ whenever $A_{ij}\not=0$. In this way we obtain a
signed graph $\G(\B)=(G,f)$ corresponding to $\B$, whose underlying
unsigned graph coincides with the underlying unoriented graph of $\Gamma$.

Note that $A$ is positive definite, since $\B$ is a companion basis.
By~\cite[Prop.\ 1.4]{bgz} or a remark in the proof of~\cite[Thm.\ 6.10]{cst},
this implies that every chordless cycle in $(G,f)$
has an odd number of positive signs.

We recall the definition of local switching from~\cite{cst}. Since we
are in the positive definite case, this takes a simpler form; see the
comment after Remark 4.6 in~\cite{cst}.

\begin{definition} \label{d:localswitching} \cite[4.3]{cst}
Let $(G,f)$ be a signed graph and suppose that the symmetric matrix
$A=A_{ij}$ with $A_{ij}=f({i,j})$ and $A_{ii}=2$ for all $i,j$ is
positive definite. Fix a vertex $k$ of $G$ and let $I$ be a subset of the
neighbours of $k$ in $G$. Let $J$ be the set of neighbours of $k$ in $G$
that do not lie in $I$. Then the \emph{local switching} of $(G,f)$ at $k$
is given by the following operations:
\begin{enumerate}
\item[(i)] Delete all edges of $G$ between $I$ and $J$.
\item[(ii)] For any $i\in I$ and $j\in J$ not originally joined in $G$, introduce
an edge $\{i,j\}$ with sign chosen so that the $3$-cycle between $i,j,k$
has an even number (i.e.\ zero or two) of positive signs.
\item[(iii)] Change the signs of all edges between $k$ and elements of $I$.
\item[(iv)] Leave all other edges and signs unchanged.
\end{enumerate}
\end{definition}

\begin{lemma} \label{l:switchismutation}
Fix a vertex $k$ of $\Gamma$. Let $I$ be
the set of vertices $i$ of $\Gamma$ for which there is an arrow
$i\rightarrow k$ in $\Gamma$, so that $J$ is the set of vertices
$j$ of $\Gamma$ for which there is an arrow $k\rightarrow j$ in $\Gamma$.
Then $(G',f')=\G(\mu_k(\B))$ coincides with the result of locally
switching $(G,f)$ at $k$ with respect to $I$.
\end{lemma}

\begin{proof}
By Proposition~\ref{p:mutatecompanionsimplylaced}, $G'$ is the
underlying unoriented graph of $\mu_k(\Gamma)$. Let $(G'',f')$ be
the result of locally switching $(G,f)$ at $k$ with respect to $I$.
Then it is easy to see from Definition~\ref{d:localswitching} that
$G''$ and $G'$ coincide as unoriented graphs.
Let $\mu_k(\B)=(\beta'_1,\ldots ,\beta'_n)$.
If $i\in I$ then
$$(\beta'_i,\beta'_k)=(s_{\beta_k}(\beta_i),\beta_k)=-(\beta_i,\beta_k)$$
and we see that the signs on the edges of $G'$ incident with $k$ and a
vertex in $I$ coincide with the signs on the corresponding edges of $G''$.
If $\{i,j\}$ is an edge in $G''$ arising from step (ii) in
Definition~\ref{d:localswitching} its sign is chosen to ensure that the
$3$-cycle on $i,j,k$ has an odd number of positive signs after step
(iii) has been applied. Since $(G',f')$ must also have this property
(using Proposition~\ref{p:cyclecondition} and the fact that $\B'$ is a
companion basis).
It is easy to see that the other signs on the edges of $G'$ and $G''$ coincide,
and the result follows.
\end{proof}

Since the signed graph $\G(\B)$ has an odd number of positive signs in every
cycle (and noting Proposition~\ref{p:opposite}) we have that the group
$\Cox(G,f)$ defined in~\cite[Defn.\ 6.2]{cst}
coincides with the group $W_{\Gamma}$ defined in Section~\ref{s:groupdefinition},
in the simply-laced case.
Then Theorem~\ref{t:isomorphism} (in the simply-laced case) follows by
combining~\cite[Thm.\ 6.10]{cst} and~\cite[Thm.\ 6.1]{parsons} (see
Proposition~\ref{p:mutatecompanionsimplylaced}) with Lemma~\ref{l:switchismutation}.

\vspace{0.4cm}

\noindent \textbf{Acknowledgements:}
This work was partly carried out while both authors were visiting the Institute
for Mathematical Research (FIM) at the ETH in Zurich. Both authors would
like to thank the FIM for its support
and, in particular, Karin Baur for her kind hospitality during the visit.
The second named author would like to thank Konstanze Rietsch for some
helpful conversations. We would also like to thank the referee for some helpful
comments, which simplified the first version of this paper.

\end{document}